%% file: main.tex
\documentclass{article}
\usepackage[SEQ]{prftree}
\usepackage{modalops}
\usepackage{mathrsfs,mathtools,amsmath,amsthm,amssymb}
\usepackage{graphicx} % Required for inserting images
\usepackage{comment}
\usepackage{hyperref}
\usepackage[capitalise,nameinlink]{cleveref}
\usepackage{dirtytalk}
\usepackage{tikz}
\usepackage{tikz-cd}
\usepackage{mathdots}
\usepackage{authblk}

\usepackage[T1]{fontenc}

\usepackage{comment}

\newtheorem{thm}{Theorem}
\newtheorem{prop}{Proposition}
\newtheorem{lem}{Lemma}
\newtheorem{cor}{Corollary}

\theoremstyle{definition}
\newtheorem{defn}{Definition}
\newtheorem{rem}{Remark}

\newtheorem{example}{Example}

\crefname{thm}{Theorem}{Theorems}
\crefname{prop}{Proposition}{Propositions}
\crefname{lem}{Lemma}{Lemmas}
\crefname{cor}{Corollary}{Corollaries}
\crefname{defn}{Definition}{Definitions}
\crefname{rem}{Remark}{Remarks}
\crefname{equation}{}{}
\crefname{section}{Section}{Sections}

\DeclareMathOperator{\LT}{LT}
\DeclareMathOperator{\Lex}{Lex}
\DeclareMathOperator{\Image}{Im}
\DeclareMathOperator{\Ch}{Ch}

\DeclareMathOperator{\Con}{Con}
\DeclareMathOperator{\Q}{Q}
\DeclareMathOperator{\Fil}{Fil}
\DeclareMathOperator{\Ker}{Ker}
\DeclareMathOperator{\Formulas}{Form}
\DeclareMathOperator{\Pro}{Pro-}
\DeclareMathOperator{\Sub}{Sub}
\DeclareMathOperator{\Norm}{Norm}
\DeclareMathOperator{\Seq}{Seq}
\DeclareMathOperator{\internal}{\delta^i}
\DeclareMathOperator{\external}{\delta^e}

\DeclareMathOperator{\height}{ht}
\DeclareMathOperator{\width}{wt}
\DeclareMathOperator{\confluent}{e}

\newcommand{\Ord}{\textup{Ord}}
\newcommand{\heart}{\ensuremath\heartsuit}
\newcommand{\KFr}{\textup{\textbf{KFr}}}
\newcommand{\PreO}{\textup{\textbf{PreO}}}
\newcommand{\Pos}{\textup{\textbf{Pos}}}
\newcommand{\GLLin}{\textup{\textbf{GL-Lin}}}
\newcommand{\MA}{\textup{\textbf{MA}}}
\newcommand{\BA}{\textup{\textbf{BA}}}

\newcommand{\CAMA}{\textup{\textbf{CAMA}}}
\newcommand{\CABA}{\textup{\textbf{CABA}}}
\newcommand{\fin}{\textit{fin}}
\newcommand{\lf}{\textit{lf}}
\newcommand{\id}{\textup{id}}
\newcommand{\op}{\textup{op}}
\newcommand{\Set}{\textup{\textbf{Set}}}

\title{A Proof Theory for Profinite Modal Algebras}
\author[1,3]{Matteo De Berardinis}
\author[2]{Silvio Ghilardi}
\affil[1]{University of Amsterdam}
\affil[2]{Università degli Studi di Milano}
\affil[3]{Università degli Studi di Salerno}

\date{March 2026}

\begin{document}

\maketitle

\begin{abstract}
In a previous paper, we showed that profinite $L$-algebras (where $L$ is a variety of modal algebras generated by its finite members) are monadic over $\bf Set$. This monadicity result suggests that profinite $L$-algebras could be presented as Lindenbaum algebras for propositional theories in  infinitary versions of  propositional modal calculi. In this paper we  identify such calculi as modal enrichments of Maehara-Takeuti's infinitary extension of the sequent calculus $\bf LK$. We also investigate correspondences between syntactic properties of the calculi and regularity/exactness properties of the opposite category of profinite $L$-algebras.
\end{abstract}

\section{Introduction}
\input{introduction}

\section{Profinite Modal Algebras}\label{sec:pre}

\input{preliminaries}

%\section{Review on trees}
\section{Infinitary Proof Theory}\label{sec:calc}

In this section we review the content of~\cref{lem:lf} in a syntactic context: this is a preliminary step to make a connection between infinitary syntax and the regularity and exactness property of the category  $\Pro L\MA_\fin^{op}\simeq L\KFr_\lf$.

Our syntax  is a modal extension of the propositional part of the calculus presented in~\cite[Ch4,\S 22]{Tak}. We  have infinite disjunctions and conjunctions, as well as infinitary sequents; however, our proofs will be well founded trees.\footnote{The only difference with the calculus in~\cite[Ch4,\S 22]{Tak} is the fact that our conjunction and disjunction connectives apply to sets of formulas and not to sequences of formulas (similarly, sequents are pairs of sets of formulas and not pairs of sequences of formulas). We preferred this simplified approach because the main aim of this paper is that of connecting infinitary syntax with lattice-based algebraic structures.}

\input{calculus}

\section{Algebraic Semantics}\label{sec:sem}

\input{algebras}

\section{Syntactic versus Categorical Properties}\label{sec:prop}

\input{bridge}

\section{Conclusions, open problems and further work}

\input{conclusions}\label{sec:conclusions}

\bibliographystyle{plain}
\bibliography{refs}
\end{document}

%% file: introduction.tex
%matteo rev
The algebraic semantics for the normal modal logic $K$ \cite{Modal} is given by \emph{modal algebras}: a modal algebra $(B,\pos)$ is a Boolean algebra $B$ endowed with a finite-join preserving operator $\pos \colon B \longrightarrow B$, i.e.\ with an operator satisfying
$$\pos (x \vee y) = \pos x \vee \pos y ~~~~~\text{and}~~~~~ \pos \bot = \bot.$$
The operator $\pos$ is called the ‘possibility’ operator and its dual ($\nec x \coloneqq \neg \pos \neg x$) is called the ‘necessity’ operator. Modal algebras and Boolean morphisms preserving $\pos$ form the category $\MA$.
We will consider varieties of modal algebras generated by their finite members. Such varieties, called 
%`finitely approximable',
\textit{finitely approximable},
are in bijective correspondences with \textit{normal modal logics} (i.e.\ logics containing $K$ and closed under \emph{uniform substitution}, \emph{modus ponens} and \emph{necessitation}) that enjoy the \textit{finite model property} (the finite model property means that any non-theorem of the logic is falsified by some finite model of its~\cite{Modal}). If $L$ is such a logic, we call $L$\emph{-algebras} its algebraic models and $L\MA$ the corresponding variety, viewed as a category. \emph{Profinite $L$-algebras} are defined as the \emph{Pro-Completion} (= the categorical construction of freely adjoining cofiltered limits to a category, see \cite{Stone}) of the full subcategory $L\MA_\fin$ of $L\MA$ given by finite $L$-algebras.
%\footnote{Here and in the whole paper, we treat normal modal logics and varieties of modal algebras interchangeably and use the letter $L$ to denote one of them. We also restrict to the case where $L$ is finitely approximable, %(namely, generated by its finite members),
%because the focus of the paper is on profinite algebras and profinite $L$-algebras are the same as profinite $L'$-algebras, where $L'$ is the subvariety of $L$ generated by its finite members.}
The category of profinite $L$-algebras is denoted by $\Pro L\MA_\fin$.
%%%%%%%%%%%%%

In~\cite{MDB-SG}, we proved that the category $\Pro \MA_\fin$ of profinite modal algebras is monadic over $\Set$ by two independent methods: by applying Beck’s monadicity theorem~\cite{CWM}, and by combining Manes’ theorem~\cite{manes} with specific facts about quotients of profinite algebras~\cite{gehrke}. This result extends immediately to the category $\Pro L\MA_\fin$ of profinite $L$-algebras.%, where $L$ is a finitely approximable variety of modal algebras.\footnote{Here and in the whole paper, we treat normal modal logics and varieties of modal algebras interchangeably and use the letter $L$ to denote one of them. We also restrict to the case where $L$ is finitely approximable (namely, generated by its finite members), because the focus of the paper is on profinite algebras and profinite $L$-algebras are the same as profinite $L'$-algebras where $L'$ is the subvariety of $L$ generated by its finite members.}

The monadicity theorem clearly indicates that profinite modal algebras are sensible to an `algebraic' description, although not in terms of a set of operations and equations of finitary character: this is because the forgetful functor associated to a finitary monad preserves filtered colimits and such a property fails for $\Pro L\MA_\fin$,
%, because the monad is not finitary: in fact algebras 
as we will show in \cref{sec:pre}. From the logician's point of view, a convenient `algebraic' description should be linked to a calculus, an infinitary calculus in our case.
Building it is precisely  the aim of the present paper. We take inspiration from Maehara-Takeuti's infinitary sequent calculus for classical logic introduced in~\cite[Chapter 4]{Tak}: in such a calculus,  conjunctions and disjunctions can have arbitrary arity, but proofs are well-founded trees. The idea is to add to it additional modal rules characterizing the specificity of profinite modal algebras.

We build our calculus in~\cref{sec:calc} and we connect it to our profinite modal algebras via a Lindenbaum construction in~\cref{sec:sem}. The calculus is parameterized by a propositional modal logic $L$, a set of propositional variables $G$ and specific theory axioms $\tau$ (to theory axioms, uniform substitution does not apply, unlike to $L$ axioms).

In~\cref{sec:prop}, we connect proof theoretic properties (basically variants of Beth Definability and Craig Interpolation Theorems) of our infinitary calculi for a logic $L$ with categorical properties of the category $\Pro L\MA_\fin$ of profinite modal $L$-algebras (better, of the dual category $L\KFr_\lf$ of locally finite Kripke frames for $L$).
We focus on the analysis of factorization systems. Being monadic,  $\Pro L\MA_\fin$ is regular and hence has a stable regular epi/mono factorization system, which becomes an epi/regular mono factorization system in $L\KFr_\lf$. The latter (the coregular factorization system of $L\KFr_\lf$) can be compared with the regular factorization system of $L\KFr_\lf$ (which also exists because $L\KFr_\lf$ is both complete and cocomplete), so one may ask the following: do such factorizations coincide? Are they stable? It turns out that the coincidence of the regular and coregular factorization in $L\KFr_\lf$ is equivalent to a strong version of the Beth Definability Property, whereas the stability under pullbacks of the coregular factorization in $L\KFr_\lf$ is equivalent to Craig Interpolation Theorem (for global consequence relation). Similar characterizations can be found for regularity of monos and of epis.  
We also supply a logical characterization of $L\KFr_\lf$ being Barr exact, together with examples and counterexamples. Thanks to Giraud's Theorem~\cite[Theorem 1.4.5]{makkaireyes}, it can be proved that $L\KFr_\lf$ is a Grothendieck topos whenever it is Barr exact.

The above analysis needs to be compared with the similar analysis from the literature on the customary finitary fragment of our language and on the customary finitary algebras. As we shall see during the paper, in many cases the results are parallel, however there are also remarkable differences. Probably the most striking difference concerns the fact that epis are always regular in $\Pro L\MA_\fin$ whenever $L$ extends the modal logic $K4$, whereas in the case of the finitary algebras this is true only if one restricts to finitely presented algebras~\cite{Mak92,Ghi-Zaw}. In general, the impression is that the framework gets simplified and proofs become smoother in the infinitary language, where for instance  the distinction between ordinary and uniform interpolants disappears and where the natural formulation of the Weak Beth Definability Property is what is usually called \say{infinite (deductive) Beth definability property}\cite{blok,Mor20}.

We conclude the paper by mentioning in~\cref{sec:conclusions} some  research directions and open problems that are specifically suggested by our infinitary framework.

%% file: preliminaries.tex
In this Section we review (and integrate) some notions and results from~\cite{MDB-SG} in order to supply the  framework of the paper.

As we said in the introduction, we consider varieties of modal algebras; these are in bijective correspondence with normal modal logics. For this reason, here and in the whole paper, we treat normal modal logics and varieties of modal algebras interchangeably and use the letter $L$ to denote one of them. The category of $L$-algebras and structure-preserving functions is denoted by $L\MA$, while $L\MA_\fin$ denotes the full subcategory given by finite $L$-algebras.
%The object of study of this paper is the class of \emph{profinite modal algebras}.
Given a logic $L$, one may consider the \emph{Pro-Completion} of the category $L\MA_\fin$; the result is the category $\Pro L\MA_\fin$, whose objects are referred to as \emph{profinite $L$-algebras}. We then restrict to the case of varieties generated by their finite members (called \emph{finitely approximable}), corresponding to normal modal logics $L$ that enjoy the \emph{finite model property}. This is because profinite $L$-algebras are the same as profinite $L'$-algebras, where $L'$ is the subvariety of $L$ generated by the finite members of $L$.

We fix the following notation for normal modal logics. Given two normal modal logics $L$ and $L'$, we say that $L'$ \emph{extends} $L$ (or that $L'$ \emph{is an extension} $L$, or that $L'$ \emph{is above} $L$) and write $L' \supseteq L$ if every formula of $L$ is contained in $L'$ (equivalently, if $L'$, viewed as a variety of algebras, is contained in $L$). Given a normal modal logic $L$ and a set of modal formulas $\Sigma$, we denote by $L + \Sigma$ the smallest normal modal logic extending $L$ that contains $\Sigma$. Since we will be working with logics enjoying the finite model property, we denote by $L +_f \Sigma$ the smallest normal modal logic with the finite model property that extends $L$ and contains $\Sigma$
%silvio 13/4/26
(as a variety, this is the variety generated by $(L+\Sigma)\MA_\fin$).
This notation allows us to describe our logics of interest by means of axioms, with the understanding that we are considering the logic of all formulas valid in the finite algebras of the given axiom system. In many 
%silvio 13/4/26
common cases, however, the operations $+$ and $+_f$  coincide.

Below, we shall recover $\Pro L\MA_\fin$ %them
%Practically speaking, we can realize this construction 
as a full subcategory of $\CAMA_\infty$, the category of \emph{complete}, \emph{atomic} and \emph{completely additive} (= $\pos$ commutes with arbitrary joins) modal algebras and Boolean morphisms preserving arbitrary joins (hence arbitrary meets), as well as $\pos$. To this aim, we first recall Goldblatt-Thomason duality of $\CAMA_\infty$ with Kripke frames.

%Throughout the paper, we will consider logics $L$ \textit{above} (or \textit{extending}) a given normal modal logic $L_0$; this means that all the theorems for $L_0$ are also theorem for $L$ or, equivalently, that all the $L$-models (either algebraic or relational) are also $L_0$-models. The category of $L$-algebras and structure-preserving functions is denoted by $L\MA$, while $L\MA_\fin$ denotes the full subcategory given by finite $L$-algebras. The objects of $\Pro L \MA_\fin$, the Pro-Completion of $L\MA_\fin$, are called profinite $L$-algebras. Below, we shall recover $\Pro L\MA_\fin$ %them
%Practically speaking, we can realize this construction 
%as a full subcategory of $\CAMA_\infty$, the category of \emph{complete}, \emph{atomic} and \emph{completely additive} (= $\pos$ commutes with arbitrary joins) modal algebras and Boolean morphisms preserving arbitrary joins (hence arbitrary meets), as well as $\pos$. To this aim, we first recall Goldblatt-Thomason duality of $\CAMA_\infty$ with Kripke frames.

%Modal logics also have a relational semantics, provided by \emph{Kripke frames}. 
A \emph{Kripke frame} $(W,\prec)$ is a set $W$ endowed with a binary relation $\prec$ on it (most of the time we will omit to indicate the binary relation). A \emph{p-morphism} from a Kripke frame $W$ to a Kripke frame $V$ is a function $f \colon W \longrightarrow V$ with the following properties:
\begin{enumerate}
    \item (stability) for $w, w' \in W$, $w \prec w' \implies f(w) \prec f(w')$;
    \item (openness) for $v' \in V$, $w \in W$, $f(w) \prec v' \implies \exists w' \in W \text{ such that } w \prec w' ~\&~ f(w') = v'$.
\end{enumerate}
Kripke frames and p-morphisms form the category $\KFr$.
%; we will denote by $L\KFr$ the full subcategory of the Kripke frames satisfying $L$. 
%
%We mention some remarkable examples. $K4$ is obtained from $K$ by adding the axiom
%$$\pos \pos x = \pos x.$$
%$K4\KFr$ is the class of transitive Kripke frames. $S4$ is obtained from $K4$ by adding the axiom
%$$x \wedge \pos x = x.$$
%$S4$-algebras are also called \emph{interior} algebras and $S4\KFr$ is the class $\PreO$ of preorders. $\Pos$ will denote the class of posets.
%
Extending Tarski duality between $\CABA$ (complete atomic Boolean algebras) and $\Set$, we obtain the following.
\begin{thm}[Goldblatt-Thomason's duality~\cite{thomason}]\label{thm:duality}
$\CAMA_\infty$  is dual to $\KFr$.
\end{thm}
\begin{proof}(Sketch)
On the modal algebras side, the duality functors make the set of atoms of a complete atomic and completely additive modal algebra into a Kripke frame, by setting $a \prec b$ if and only if $a \leq \pos b$. On the other side, for a Kripke frame $W$, the Boolean algebra $\mathcal{P}(W)$ is turned into a modal algebra by setting $\pos X := \{w \in W\ \vert\ \exists w' (w \prec w' ~\&~ w' \in X)\}$ or, dually, $\nec X := \{w \in W\ \vert\ \forall w' (w \prec w' \implies w' \in X)\}$.
\end{proof}
%matteo rev
Kripke frames provide a relational semantics for modal logics. An assignment of subsets of a frame $W$ to propositional variables can be extended to all modal formulas in the following way: Boolean connectives are interpreted point-wise, whereas a formula of the form $\nec \varphi$ is satisfied at $w \in W$ if each $w' \in W$ such that $w \prec w'$ satisfies $\varphi$ (see \cite{Modal}). %It is then clear that
%%%%%%%%%%%%%%
Goldblatt-Thomason's duality then restricts to useful full subcategories: given a logic $L$, one may consider the full subcategory $L\CAMA_\infty$ formed by the complete, atomic and completely additive modal algebras belonging to $L\MA$; its dual is the category $L\KFr$, whose objects are precisely those Kripke frames validating $L$ (in the semantics mentioned above). %, whose dual power set modal algebra is in $L\CAMA_\infty$.
Since finite modal algebras are complete, atomic and completely additive, the category of finite modal $L$-algebras $L\MA_\fin$ is dual to $L\KFr_\fin$, the full subcategory of $L\KFr$ consisting of its finite Kripke frames.

Notable examples of normal modal logics with the finite model property %matteo rev
%\footnote{As we said in the introduction, $L$ has the finite model property if every formula $\varphi$ that does not belong to $L$ fails in some finite model of $L$ (either algebraic or Kripke, since there is no distinction, by Goldblatt-Thomason's duality restricted to finite objects).}
%%%%%%%%
are the logics $K4$, $S4$, \emph{GL}, \emph{Grz}, \emph{GL.Lin}, \emph{Grz.Lin}, etc: the Table below reports the axiomatization of such logics and the characterization of their finite Kripke frames (see~\cite{Modal} for more information and for the involved proofs).

\begin{table}\label{table:logics}
\begin{tabular}{||l|l|l||}  \hline
Logic  $L$       &    Axioms      & Corresponding finite  \\
 %                &                & of the  \\
                 &                &  Kripke frames \\ \hline
 $K$        & ~~~~~~~  - &~~~~~~ - \\ \hline
 $K4$        &  $K$ + $\nec x \to \nec \nec x$ & $\prec$ is transitive  \\
                     \hline
                % &                             & and transitive   \\
 $S4$        &   $K4$ + $\nec x \to x$ & $\prec$ is reflexive and \\
                 &                             &  transitive   \\
                                         \hline
 %{\bf S4.2}      &   {\bf S4} +  & reflexive and \\
 %                &    $\pos \nec p \to \nec \pos p$        &  locally confluent \\ \hline
 $S4$\emph{.Lin}      &   $S4$ +  & $\prec$ is reflexive, transitive  \\
                 &    $\nec (\nec x_1 \to x_2) \vee \nec (\nec x_2 \to x_1)$  & and locally linear  \\ \hline
 \emph{Grz}      &   $S4$ +  & $\prec$ is a partial order \\
                 &    $(\nec (\nec (x \to \nec x) \to x) \to x)$ &  \\ \hline
 \emph{Grz.Lin}      &   $S4$\emph{.Lin} + \emph{Grz}  & $\prec$ is a locally linear \\
                 &     & partial order \\ \hline
 \emph{GL}        &   $K4$ + $(\nec (\nec x \to x) \to \nec x)$ & $\prec$ is irreflexive and \\
                  &                                               &  transitive    \\
                                             \hline
  \emph{GL.Lin}        &   \emph{GL} +  & $\prec$ is irreflexive,  \\
                  &      $\nec^+ (\nec^+ x_1 \to x_2) \vee \nec^+ (\nec^+ x_2 \to x_1)$       & transitive and \\
    &   & locally linear \\
                                             \hline
\end{tabular}
\caption{$\nec^+ x$ stands for $x \wedge \nec x$; a relation $\prec$ is said to be \emph{locally linear} if and only if it satisfies the condition $(w\prec v_1 ~\&~ w\prec v_2) \Rightarrow ( v_1\prec v_2 \vee v_1=v_2 \vee v_2\prec v_1)$ for all $v,w_1,w_2$.}
\end{table}

By Goldblatt-Thomason's duality, in order to study the Pro-Completion of $L\MA_\fin$, we can describe the dual construction, called \emph{Ind-Completion}, over the category $L\KFr_\fin$, which gives the following class of Kripke frames.
\begin{defn}
A Kripke frame $W$ is said to be \emph{locally-finite} if, for all $w \in W$,
$$w^* \coloneqq \{w' \in W\ \vert\ \exists w_0,\dots,w_n \in W \text{ such that } w=w_0 \prec \dots \prec w_n=w'\}$$
is a finite set. In other words, fixing any starting point $w \in W$ and taking finite walks along the relation $\prec$, one can visit only finitely many $w' \in W$.
\end{defn}
We denote by $L\KFr_\lf$ the full subcategory of $L\KFr$ consisting of its locally-finite Kripke frames.

%\begin{rem}
%The class $L\KFr_\lf$ ($L\KFr_\fin$) is closed under (finite) disjoint unions, p-morphic images and generated subframes.
%\end{rem}

%We give some results whose proofs can be found in \cite{MDB-SG}.

\begin{prop}[\cite{MDB-SG}]\label{prop:colim}
$L\KFr$ is cocomplete; colimits are inherited by the full subcategory $L\KFr_\lf$ and preserved by the forgetful functor to $\Set$. 
\end{prop}

\begin{thm}[\cite{MDB-SG}]\label{thm:ind}
$L\KFr_{\lf}$ is the Ind-Completion of $L\KFr_{\fin}$ via the full inclusion $L\KFr_{\fin}\subseteq L\KFr_{\lf}$.
\end{thm}

As a consequence,
\begin{cor}[\cite{MDB-SG}]\label{cor:lim}
$L\KFr_\lf$ has all limits.
\end{cor}

Dually we have:
\begin{thm}\label{thm:produality} The Goldblatt-Thomason duality restricts to a duality between the category of profinite modal $L$-algebras $\Pro L\MA_\fin$ (defined as the Pro-Completion of $L\MA_\fin$) and the category  $L\KFr_\lf$.
\end{thm}

From the above picture, it follows that $\Pro L\MA_\fin$ is complete and co-complete; the forgetful functor from $L\CAMA_\infty$ to $\Set$ preserves limits, but does not have an adjoint; on the contrary, its restriction to $\Pro L\MA_\fin$  does have an adjoint and is also \emph{monadic}, as proved in~\cite{MDB-SG}.
%matteo rev in two ways, either by a direct application of Beck theorem~\cite{CWM} or by combining Manes monadicity theorem~\cite{manes} with specific facts about quotients of profinite algebras~\cite{gehrke}.

The monadicity result suggests that it is possible to characterize the members of $\Pro L\MA_\fin$ `algebraically', possibly using equations involving infinitary meets and joins, as happens with complete atomic Boolean algebras. Below we show that this is indeed the case and this motivates the introduction of the infinitary calculi of next section and the study of the correspondence between proof-theoretic properties of such calculi and categorical properties of $\Pro L\MA_\fin$.

To proceed, we need some notation. Given a set of sets $\mathscr{X}$, $\Ch(\mathscr{X})$ is the set of \emph{choice functions} over $\mathscr{X}$, i.e.\ functions $c \colon \mathscr{X} \longrightarrow \bigcup \mathscr{X}$ such that $c(X) \in X$ for each $X \in \mathscr{X}$. For a set $X$ and a cardinal $\kappa$, $\mathcal{P}(X)$ and $\mathcal{P}_\kappa(X)$ denote the set of subsets of $X$ and the set of subsets of $X$ of cardinality less than $\kappa$, respectively.
We first recall the following.
 
\begin{lem}[{\cite[Chapter VII, Proposition 1.16]{Stone}}]
The following conditions are equivalent for a Boolean algebra $B$:
\begin{enumerate}\label{lem:caba}
    \item[\rm (i)] $B$ is complete and atomic;
    \item[\rm (ii)] $B$ is isomorphic to a powerset Boolean algebra;
    \item[\rm (iii)] $B$ is complete and satisfies the infinitary distributive law
    \begin{align}\label{eq:id}
    \bigwedge \left\{\bigvee X\ \middle\vert\ X \in \mathscr{X}\right\} = \bigvee \left\{\bigwedge \{d(X)\ \vert\ X \in \mathscr{X}\}\ \middle\vert\ d \in \Ch(\mathscr{X})\right\}\tag{ID}
    \end{align}
    for every family $\mathscr{X}$ of subsets of $B$.
\end{enumerate}
\end{lem}

\begin{proof} (Sketch) If a Boolean algebra is complete and atomic, then it is isomorphic to the power set of its atoms. That a powerset Boolean algebra satisfies the infinitary distributive law above is clear. If the infinitary distributivity law holds in a Boolean algebra $B$, then the algebra is atomic and its atoms are the nonzero elements obtained by applying the infinite distributivity law to the first member of the equation $\bigwedge_{x\in B} (x\vee \neg x)=1$.
\end{proof}

%silvio Nov2025
Below we indicate with $\mathcal{P}_\omega(X)$ the set of the finite subsets of a set $X$.

\begin{lem}\label{lem:lf}
The following conditions are equivalent for a modal algebra $M$:
\begin{enumerate}
    \item[\rm (i)] $M$ is profinite;
    \item[\rm (ii)] $M$ is isomorphic to the powerset of a locally-finite Kripke frame $W$ (in the sense of the Goldblatt-Thomason's duality: see \cref{thm:duality} for the definition of $\nec$ and $\pos$ on $\mathcal{P}(X)$);
    \item[\rm (iii)] $M$ is complete, atomic, completely additive (i.e.\
    \begin{align}\label{eq:ca}
    \pos \bigvee X = \bigvee \pos X\tag{CA}
    \end{align}
    holds for each subset $X \subseteq M$) and satisfies
    %matteo 17/11/2025 (the proof also needs to be modified)
    \begin{align}\label{eq:lf}
    \bigwedge_{n\in \omega} \left\{\nec^n \bigvee X_n\right\} \leq \bigvee \left\{ \bigwedge_{n\in \omega} \left\{ \nec^n \bigvee Y\right\}\ \middle\vert\ Y \in \mathcal{P}_\omega\left(\bigcup X_n\right) \right\}\tag{LF}
    \end{align}
    %\begin{align}\label{eq:lf}
    %\bigwedge_{n\in \omega} \left\{\nec^n \bigvee X_n\right\} = \bigvee \left\{ \bigwedge_{n\in \omega} \left\{ \nec^n \bigvee (Y \cap X_n)\right\}\ \middle\vert\ Y \in \mathcal{P}_\omega\left(\bigcup X_n\right) \right\}\tag{LF}
    %\end{align}
    for every family $\{X_n\ \vert\ n \in \omega\}$ of subsets of $M$ ($\nec^n$ and $\pos^n$ denote the application, iterated $n$-times, of $\nec$ and $\pos$, respectively).
\end{enumerate}
\end{lem}

\begin{proof}
In view of the previous results and of the previous lemma, we only need to check that a Kripke frame $W$ is locally finite if and only if its dual powerset modal algebra $\mathcal{P}(W)$ satisfies \cref{eq:lf}.\footnote{
We thank J. Marqu\`es for pointing out a problem in our early formulation of condition \cref{eq:lf}.
} 
%silvio Nov25
%matteo 17/11/2025 (notice that the $\geq$-part of~\eqref{eq:lf} is trivially valid).  

Suppose that $W$ is locally finite and consider subsets $X_n \subseteq \mathcal{P}(W)$, for $n \in \omega$. Take $w\in W$ and suppose that $w \in \bigwedge_{n \in \omega} \{\nec^n \bigvee X_n\}$, i.e.\ that $w\in \nec^n \bigvee X_n$ for all $n\in \omega$. For every $n \in \omega$, let $w^n = \{w_{n,1}, \dots, w_{n,k_n}\}$ be the set of the elements of $w^*$ which are $\prec$-reachable from $w$ in $n$ steps
%matteo 17/11/2025
and not reachable in less than $n$ steps.
%matteo rev1
By $w \in \nec^n \bigvee X_n$, we have that, for each $i = 1,\dots,k_n$, $w_{n,i} \in \bigvee X_n$; this means that we can find $a_{n,i} \in \mathcal{P}(W)$, such that $a_{n,i} \in X_n$ and $w_{n,i} \in a_{n,i}$.
%
%matteo rev1 Let $Y$ contain exactly one element $a_{n,j}\in X_n$ such that $w_{n,j}\in a_{n,j}$, for each $n \geq 0$, $j = 1,\dots,k_n$.
Let $Y \coloneqq \{a_{n,i}\ \vert\ n \in \omega,\ i \in \{1,\dots,k_n\}\}$. Observe that $Y$ is finite,
%matteo 17/11/2025
$w^n$ being finite for each $n$ and definitely equal to the empty set (by the locally finiteness assumption). Moreover, $Y \subseteq \bigcup X_n$.
It is clear that 
%matteo 17/11/2025 $w\in \bigwedge_{n\in \omega} \left\{ \nec^n \bigvee (Y \cap X_n)\right\}$.
$w \in \nec^n \bigvee Y$, for all $n \in \omega$.

Vice versa, suppose that \cref{eq:lf} holds in $\mathcal{P}(W)$. Take $w\in W$ and let $X_n$ be equal to the
%matteo 17/11/2025 elements of $w^*$ which are $\prec$-reachable from $w$ in $n$ steps.
set $\mathcal{P}_\omega(w^*)$ of finite subsets of $w^*$, for each $n \in \omega$.
Then we have  $w \in \nec^n \bigvee X_n$ for all $n\in \omega$ %matteo 17/11/2025
(since $\bigvee X_n = w^*$),
thus by \cref{eq:lf} there is a finite set $Y\subseteq \bigcup X_n %matteo 17/11/2025
= \mathcal{P}_\omega(w^*)
$
such that, for all $n\in \omega$, we have
%matteo 17/11/2025 $w\in \nec^n \bigvee (Y \cap X_n)$.
$w \in \nec^n \bigvee Y$.
Then  every element $\prec$-reachable from $w$ in $n$ steps must be in
%matteo 17/11/2025 the finite set $Y_n:=Y \cap X_n$
the finite set $\bigvee Y$ ($Y$ is a finite collection of finite sets)
and so $w^*$ is finite.
%matteo 17/11/2025 because also the union of the $Y_n$ is finite. 
\end{proof}

%matteo rev
Let $M$ be a profinite modal algebra. The following consequences of \cref{eq:lf} will be restated in their proof theoretic version in \cref{sec:sem} and exploited in order to obtain a Lindenbaum construction for profinite modal algebras.
\begin{rem}\label{rem:1}
Let $X \subseteq M$. Set $X_1 \coloneqq X$ and $X_n \coloneqq X \cup \{\neg \bigvee X\}$, for $n \neq 1$. If $n \neq 1$, then $\bigvee X_n = \top$ (the top element of $M$). 
%$\bigvee X_n = M$. 
As a consequence, $\nec \bigvee X = \bigwedge_{n\in \omega} \left\{\nec^n \bigvee X_n\right\} \wedge \nec \bigvee X$. Applying \cref{eq:lf} together with $\bigwedge_{n\in \omega} \left\{ \nec^n \bigvee Y\right\} \leq \nec \bigvee Y$, \cref{eq:id} and the dual of \cref{eq:ca}, we obtain
\begin{align*}
\nec \bigvee X &\leq \bigvee \left\{ \nec \bigvee Y\ \middle\vert\ Y \in \mathcal{P}_\omega\left(\bigcup X_n\right) \right\} \wedge \nec \bigvee X =\\
&= \bigvee \left\{ \nec \bigvee Y \wedge \nec \bigvee X\ \middle\vert\ Y \in \mathcal{P}_\omega\left(\bigcup X_n\right) \right\} =\\
&= \bigvee \left\{ \nec \left(\bigvee Y \wedge \bigvee X\right)\ \middle\vert\ Y \in \mathcal{P}_\omega\left(\bigcup X_n\right) \right\} =\\
&= \bigvee \left\{ \nec \bigvee \left\{y \wedge \bigvee X\ \vert\ y \in Y\right\}\ \middle\vert\ Y \in \mathcal{P}_\omega\left(\bigcup X_n\right) \right\}.
\end{align*}
Now, using that $Y \subseteq \bigcup X_n = X \cup \{\neg \bigvee X\}$, we obtain
\begin{align*}\label{eq:lfzero}
\nec \bigvee X \leq \bigvee \left\{ \nec \bigvee Y\ \middle\vert\ Y \in \mathcal{P}_\omega(X) \right\}.\tag{$\text{LF}_0$}
\end{align*}
\end{rem}

\begin{rem}
Let $X \subseteq M$ be \textit{directed}: for each $x, x' \in X$, there exists $x'' \in X$ such that $x \leq x''$ and $x' \leq x''$ (equivalently, we can say that, for each finite subset $Y \subseteq X$, there exists $x \in X$ such that $\bigvee Y \leq x$, since arbitrary joins are available in $M$). As a consequence, $\bigvee \left\{ \nec \bigvee Y\ \middle\vert\ Y \in \mathcal{P}_\omega(X) \right\} \leq \bigvee \nec X$. Then, by \cref{eq:lfzero}, we have that $\nec \bigvee X \leq \bigvee \nec X$, hence
$$\nec \bigvee X = \bigvee \nec X,$$
the other inequality being straightforward.
\end{rem}

\begin{rem}\label{rem:important}
Let $X \subseteq M$. We know that $\pos \bigvee X = \bigvee \pos X$, by \cref{eq:ca}. If $X$ is directed, then, by the above remark, $\nec \bigvee X = \bigvee \nec X$. Notice that, for any $X \subseteq M$, we have that $\bigvee X = \bigvee X'$, where $X'$ is the directed subset obtained by adding to $X$ all \emph{finite} joins of elements from $X$.
%there exists $X' \subseteq M$ directed such that $\bigvee X = \bigvee X'$ (it is enough to take $X' \coloneqq X \cup \{\bigvee X\}$). 
Let now $G$ be a subset of $M$ that generates $M$ as a complete modal algebra; call $M_0$ the sub-modal algebra of $M$ generated by $G$. Then, all the elements of $M$ are expressible in the form
$$\bigwedge \left\{ \bigvee X\ \middle\vert\ X \in \mathscr{X}\right\},$$
where $\mathscr{X}$ is a family of subsets of $M_0$.
\end{rem}

The above remark is very important and marks the difference with the case of complete atomic modal algebras (which are not monadic over \Set, simply because there the forgetful functor does not have an adjoint~\cite{MDB-SG}). What the above remark says is that \emph{we can reach all the elements of a profinite modal algebra by applying (in this order) finitary operations, then joins and then meets to any set of generators $G$}, so that in particular the cardinality of the algebra does not exceed $2^{2^{\kappa}}$, where $\kappa$ is the maximum between $\omega$ and the cardinality of $G$. This ensures that the ``free algebras'' exist and that their elements can be represented by a \emph{set} of suitable `infinitary' terms.

We conclude this section by showing that indeed infinitary operations are strictly needed: in fact, a presentation of $\Pro L\MA_\fin$ in terms of finitary operations would imply that the monad of profinite $L$-algebras is \textit{finitary} (see \cite{AR}). The following result shows that this cannot be the case. This observation reinforces the introduction of infinitary calculi in the next section.
\begin{prop}
The monad associated to the forgetful functor $\Pro L\MA_\fin \longrightarrow \Set$ is not finitary.
\end{prop}
\begin{proof}
It is a well known fact that the forgetful functor of the algebras for a finitary monad over $\Set$ creates, hence preserves, filtered colimits (see \cite[\S 2.7]{AR}).

Let us fix an infinite countable set $X$ and consider the filtered diagram of the inclusions of the finite subsets $X_0 \subseteq X$. Obviously, the inclusions $X_0 \longrightarrow X$ form a universal cone in $\Set$. Applying the contravariant power-set functor $\mathcal{P} \colon \Set^\op \longrightarrow \Set$, we obtain the cofiltered diagram of the functions $\mathcal{P}(X'_0) \longrightarrow \mathcal{P}(X_0)$ sending $Y \in \mathcal{P}(X'_0)$ to $Y \cap X_0 \in \mathcal{P}(X_0)$, whenever $X_0$ and $X'_0$ are finite subsets of $X$ and $X_0 \subseteq X'_0$. Applying $\mathcal{P}$ a second time gives us the filtered diagram of the functions $\mathcal{P} \mathcal{P} (X_0) \longrightarrow \mathcal{P} \mathcal{P} (X'_0)$ sending a family $\mathscr{X} \in \mathcal{P} \mathcal{P}(X_0)$ of subsets of $X_0$ to the family $\{Y \in \mathcal{P}(X'_0)\ \vert\ Y \cap X_0 \in \mathscr{X}\}$ of subsets of $X'_0$.

Since $X_0$ is finite, $\mathcal{P} \mathcal{P}(X_0)$, with its set-theoretic operations, is the free Boolean algebra $F_{\BA}(X_0)$ generated by $X_0$. The free Boolean algebra functor $F_{\BA}$ preserves all colimits, being a left adjoint. Moreover, the forgetful functor from the category of Boolean algebras to $\Set$ preserves filtered colimits, since Boolean algebras are algebras for a finitary monad. Putting everything together, the colimit in $\Set$ of the filtered diagram $\mathcal{P} \mathcal{P} (X_0) \longrightarrow \mathcal{P} \mathcal{P} (X'_0)$ is the underlying set of the free Boolean algebra $F_{\BA}(X)$ generated by $X$. Observe that $|F_{\BA}(X)| = |X|$, since $F_{\BA}(X)$ is a quotient of the set of Boolean terms in the variables $X$ and $X$ is infinite countable.

Now, we want to give to the finite sets $\mathcal{P} \mathcal{P}(X_0)$ a structure of $L$-algebra, so that the diagram $\mathcal{P} \mathcal{P} (X_0) \longrightarrow \mathcal{P} \mathcal{P} (X'_0)$ can be seen as living in $\Pro L\MA_\fin$. By a classical result due to Makinson (see also \cite[Prop. 3.17]{Modal}), every normal modal logic $L$ is contained either in the logic of the reflexive one-point frame or in that of the irreflexive one-point frame. In the first case, we make the set $\mathcal{P}(X_0)$ into a Kripke frame, by equipping it with the equality relation; dually, $\pos \colon \mathcal{P} \mathcal{P}(X_0) \longrightarrow \mathcal{P} \mathcal{P}(X_0)$ is the identity. In the second case, we equip $\mathcal{P}(X_0)$ with the empty relation; dually, $\pos$ is the constant map with value $\bot$. In both cases, we have $\mathcal{P}(X_0) \in L\KFr_\fin$ and $\mathcal{P} \mathcal{P}(X_0) \in L\MA_\fin$; moreover, the functions $\mathcal{P}(X'_0) \longrightarrow \mathcal{P}(X_0)$ are p-morphisms, hence the diagram $\mathcal{P} \mathcal{P} (X_0) \longrightarrow \mathcal{P} \mathcal{P} (X'_0)$ can be seen in $L\MA_\fin$. We can do the same with $X$, obtaining a cone $\mathcal{P}(X) \longrightarrow \mathcal{P}(X_0)$ in $L\KFr_\lf$ (despite $\mathcal{P}(X)$ being an infinite set, it is locally finite as a Kripke frame) and a cocone $\mathcal{P} \mathcal{P} (X_0) \longrightarrow \mathcal{P} \mathcal{P}(X)$ in $\Pro L\MA_\fin$. The cone $\mathcal{P}(X) \longrightarrow \mathcal{P}(X_0)$ in $L\KFr_\lf$ is universal: given $W \in L\KFr_\lf$ and a cone $f_{X_0} \colon W \longrightarrow \mathcal{P}(X_0)$ in $L\KFr_\lf$, the p-morphism $f \colon W \longrightarrow \mathcal{P}(X)$ witnessing the universality is given by
$$f(w) \coloneqq \{x \in X\ \vert\ f_{\{x\}}(w) \neq \emptyset\} \in \mathcal{P}(X),$$
for each $w \in W$. Dually, the cocone $\mathcal{P} \mathcal{P} (X_0) \longrightarrow \mathcal{P} \mathcal{P}(X)$ in $\Pro L\MA_\fin$ is universal.

We proved that $\mathcal{P} \mathcal{P}(X)$ is the result of computing the colimit of the diagram $\mathcal{P} \mathcal{P} (X_0) \longrightarrow \mathcal{P} \mathcal{P} (X'_0)$ in $\Pro L\MA_\fin$. As we previously said, if the same colimit is computed in $\Set$, we obtain (the underlying set of) $F_{\BA}(X)$. However, $|F_{\BA}(X)| = |X| < |\mathcal{P} \mathcal{P} (X)|$. This means that the forgetful functor $\Pro L\MA_\fin \longrightarrow \Set$ does not preserve filtered colimits, hence the associated monad cannot be finitary.
\end{proof}

%% file: calculus.tex
\subsection{Formulae}

%In what follows, we provide a sequent calculus for the infinitary version of a modal logic $L$ axiomatizing a finitely approximable variety (generated by its finite members: it corresponds to $L$ having the finite model property) of modal algebras. The language is an extension of the standard one, enriched with infinitary disjunction and conjunction symbols, to be applied to sets of formulas (of bounded cardinality). \textcolor{red}{As we said, a rule modifies infinitely many formulas...}

Fix a set $G$ and let $\kappa$ be a regular cardinal. We want to consider those formulas that can be generated from $G$, using the standard connectives of propositional modal logic, together with infinitary conjunctions and disjunctions that allow to connect any family of less than $\kappa$ formulas. As usual, the formulas are defined inductively, so that a notion of complexity naturally arises. For each ordinal $\alpha$, consider the set $\Formulas_{\kappa,\alpha}(G)$ of $\kappa$-\emph{ary} formulas of \emph{height} (less or equal than) $\alpha$ over the propositional variables $G$, namely
\begin{enumerate}
    \item $\Formulas_{\kappa,0}(G) \coloneqq G$;
    \item $\Formulas_{\kappa,\alpha + 1}(G) \coloneqq \Formulas_{\kappa,\alpha}(G) \cup \Formulas_{\kappa,\alpha}^+(G)$, where $\Formulas_{\kappa,\alpha}^+(G)$ is the set
    $$\{\neg \varphi,\ \nec \varphi\ \vert\ \varphi \in \Formulas_{\kappa,\alpha}(G)\} \cup \left\{\bigwedge S,\ \bigvee S\ \middle\vert\ S \subseteq \Formulas_{\kappa,\alpha}(G),\ |S| < \kappa \right\}$$
    \item if $\beta$ is a limit ordinal, then $\Formulas_{\kappa,\beta}(G) \coloneqq \bigcup_{\alpha < \beta} \Formulas_{\kappa,\alpha}(G)$.
\end{enumerate}

Observe that $\alpha \leq \beta$ implies $\Formulas_{\kappa,\alpha}(G) \subseteq \Formulas_{\kappa,\beta}(G)$. Moreover, if $\alpha \geq \kappa$, then $\Formulas_{\kappa,\kappa}(G) = \Formulas_{\kappa,\alpha}(G)$ (this is established by transfinite induction on $\alpha$, using regularity of $\kappa$), so that
$$\Formulas_\kappa(G) \coloneqq \bigcup_{\alpha \in \Ord} \Formulas_{\kappa,\alpha}(G) = \Formulas_{\kappa,\kappa}(G)$$
is a set (called the set of \emph{$\kappa$-ary formulas over the variables $G$}). $\Formulas_\omega(G)$ can be identified with the set of standard finitary formulas over $G$.

%From now on, $L$ will denote a logic, say in the countable set of propositional variables $\Prop = \{p,q,r,\dots\}$, axiomatizing a finitely approximable variety of modal algebras. We can identify the axioms of $L$ with a subset of $\Formulas_\omega(\Prop)$.

We have the following standard notation:
\begin{enumerate}
    \item $\varphi \wedge \psi \coloneqq \bigwedge \{\varphi,\psi\}$ and $\top:=\bigwedge \{\emptyset\}$ (similarly for $\vee, \bot$);
    \item $\varphi \rightarrow \psi \coloneqq \neg \varphi \vee \psi$;
    \item $\varphi \leftrightarrow \psi \coloneqq (\varphi \rightarrow \psi) \wedge (\psi \rightarrow \varphi)$;
    \item $\pos \varphi \coloneqq \neg \nec \neg \varphi$;
    \item $\nec^n \varphi \coloneqq \nec \cdots \nec \varphi$ $n$-times, for $n \in \omega$, with $\nec^0 \varphi \coloneqq \varphi$ (similarly for $\pos^n$);
    \item $\nec^* \varphi \coloneqq \bigwedge \{\nec^n \varphi\ \vert\ n \in \omega\}$ (similarly for $\pos^*$).
\end{enumerate}

In the following, we use the letter $\varphi$ for a generic formula in $\Formulas_\kappa(G)$ and $\Gamma$, $\Gamma'$, $\Delta$, $\Delta'$ (as well as $S, S',\cdots$) for subsets of $\Formulas_\kappa(G)$. An expression of the form $\Gamma \Rightarrow \Delta$ is called a \emph{sequent}. When we write sets of formulas  separated by a comma, we mean their set-union (for singletons, we shall avoid parentheses).

\subsection{Rules}

We now give the rules of the sequent calculus for our logic.
The rules are obtained from the rules in~\cite[Ch.4,\S 22]{Tak}, by adding specific rules for modalities. Notice that, unlike in customary finitary versions of the sequent calculus, a rule may modify simultaneously many formulae of the lower sequent and a rule may have infinitely many premises: this is needed to keep our proofs well founded.
Similarly, in the Cut Rule, infinitely many formulas $\varphi \in S$ may be simultaneously cut away. 
In~\cite[Ch.4, Thm. 22.17]{Tak} a semantic cut elimination proof for the non-modal part of the calculus below is supplied; finding a cut elimination proof for a suitable reformulation of our calculus is a major open problem.

Rules must be read from top to bottom (however, most of them are invertible): each rule is given by a set of sequents $\{\Gamma_i \Rightarrow \Delta_i\ \vert\ i \in I\}$, called the set of \textit{premises}, and by an additional sequent $\Gamma \Rightarrow \Delta$, called the \textit{conclusion}, and it is displayed as
\begin{align*}
\prftree[r]{\scriptsize
\text{$\begin{matrix}
    \text{N} \\
    (i \in I)
\end{matrix}$}}
{\cdots}{\Gamma_i \Rightarrow \Delta_i}{\cdots}
{\Gamma \Rightarrow \Delta}
\end{align*}
where N indicates the name. When we write $\cdots \Gamma_i \Rightarrow \Delta_i \cdots;\ \overline{\Gamma} \Rightarrow \overline{\Delta}$, we mean that the set of premises is $\{\Gamma_i \Rightarrow \Delta_i\ \vert\ i \in I\} \cup \{\overline{\Gamma} \Rightarrow \overline{\Delta}\}$.

\vskip .5cm

Axioms:

\begin{align*}
\prftree[r]{\scriptsize Ax}
{}
%{\varphi, \Gamma \Rightarrow \Delta, \varphi}
{\varphi \Rightarrow \varphi}
\end{align*}

Structural Rules:

\begin{align*}
\prftree[r]{\scriptsize W}
{\Gamma \Rightarrow \Delta}
{\Gamma, \Gamma' \Rightarrow \Delta, \Delta'}
\end{align*}

\begin{align*}
\prftree[r]{\scriptsize
\text{$\begin{matrix}
    \text{Cut} \\
    %(\varphi \in S)
    (S=\{\varphi_i\}_i)
\end{matrix}$}}
%Cut}{\scriptsize (\text{$\varphi \in S$})}
{\cdots}{\Gamma \Rightarrow \Delta, \varphi_i}{\cdots;}{S, \Gamma' \Rightarrow \Delta'}
{\Gamma, \Gamma' \Rightarrow \Delta, \Delta'}
\end{align*}
for some $S \subseteq \Formulas_\kappa(G)$.\footnote{If $S = \{\varphi\}$, we obtain the standard Cut rule, with $\Gamma \Rightarrow \Delta, \varphi$ and $\varphi, \Gamma' \Rightarrow \Delta'$ as premises.}
%for some $S \subseteq \Formulas_\kappa(G)$.

Inference Rules:

\begin{align*}
\prftree[r]{\scriptsize L\text{$\neg$}}
{\Gamma \Rightarrow \Delta, S}
{\neg S, \Gamma \Rightarrow \Delta}
\qquad
\prftree[r]{\scriptsize R\text{$\neg$}}
{S, \Gamma \Rightarrow \Delta}
{\Gamma \Rightarrow \Delta, \neg S}
\end{align*}
for some $S \subseteq \Formulas_\kappa(G)$, with $\neg S \coloneqq \{\neg \varphi\ \vert\ \varphi \in S\}$.

\begin{align*}
\prftree[r]{\scriptsize L\text{$\bigwedge$}}
{\bigcup \mathscr{S}, \Gamma \Rightarrow \Delta}
{\left\{\bigwedge S\ \middle\vert\ S \in \mathscr{S} \right\}, \Gamma \Rightarrow \Delta}
\qquad
\prftree[r]{\scriptsize
\text{$\begin{matrix}
    \text{R}\!\bigwedge \\
    (c \in \Ch(\mathscr{S}))
\end{matrix}$}}
{\cdots}{\Gamma \Rightarrow \Delta, \{c(S)\ \vert\ S \in \mathscr{S} \}}{\cdots}
{\Gamma \Rightarrow \Delta, \left\{\bigwedge S\ \middle\vert\ S \in \mathscr{S} \right\}}
\end{align*}
for some $\mathscr{S} \subseteq \mathcal{P}(\Formulas_\kappa(G))$ such that  $|S|<\kappa$ for each $S \in \mathscr{S}$.

\begin{align*}
\prftree[r]{\scriptsize
\text{$\begin{matrix}
    \text{L}\!\bigvee \\
    (c \in \Ch(\mathscr{S}))
\end{matrix}$}}
{\cdots}{\{c(S)\ \vert\ S \in \mathscr{S} \}, \Gamma \Rightarrow \Delta}{\cdots}
{\left\{\bigvee S\ \middle\vert\ S \in \mathscr{S} \right\}, \Gamma \Rightarrow \Delta}
\qquad
\prftree[r]{\scriptsize R\text{$\bigvee$}}
{\Gamma \Rightarrow \Delta, \bigcup \mathscr{S}}
{\Gamma \Rightarrow \Delta, \left\{\bigvee S\ \middle\vert\ S \in \mathscr{S} \right\}}
\end{align*}
for some $\mathscr{S} \subseteq \mathcal{P}(\Formulas_\kappa(G))$ such that  $|S|<\kappa$ for each $S \in \mathscr{S}$.

\begin{align*}
\prftree[r]{\scriptsize Nec}
{S \Rightarrow \varphi}
{\nec S \Rightarrow \nec \varphi}
\end{align*}
for some $S \subseteq \Formulas_\kappa(G)$, with $\nec S \coloneqq \{\nec \varphi\ \vert\ \varphi \in S\}$.

\begin{align*}
\prftree[r]{\scriptsize
\text{$\begin{matrix}
    \text{lf} \\
    (I \in \mathcal{P}_\omega\left(\bigcup_n S_n\right))
\end{matrix}$}}
{\cdots}
{\left\{ \nec^n \bigvee
%matteo 17/11/2025 (I \cap S_n)
I\ %
\middle\vert\ n \in \omega\right\}, \Gamma \Rightarrow \Delta}
{\cdots}
{\left\{\nec^n \bigvee S_n\ \middle\vert\ n \in \omega\right\}, \Gamma \Rightarrow \Delta}
\end{align*}
for some $S_n \subseteq \Formulas_\kappa(G)$ such that $|S_n|<\kappa$, for $n \in \omega$.

There are additional rules 
%sg july25
(see below)
in case the calculus comprises a normal modal logic (with finite model property) $L$ and a theory $T$. A theory $T$ is just a set of formulas in $\Formulas_\kappa(G)$, called \emph{theory axioms}.

On the other hand, as we said in the introduction, the logics we are interested in are defined in correspondence with varieties of modal algebras. This means that they are axiomatized via axiom schemata in a countable language; consequently, a logic $L$ gives rise to the additional axioms set given by the formulae $L_{(G,\kappa)}$ obtained from some $\chi(p_1, \dots, p_n) \in L$, by replacing the propositional variables $p_1,\dots,p_n$ by some $\varphi_1,\dots,\varphi_n \in \Formulas_\kappa(G)$. We call the formulae in $L_{(G,\kappa)}$ \emph{logical axioms} (and we may  occasionally write $(G,\kappa,L,T)$ simply as $L$).

When a logic $L$ and a theory $T$ are fixed, the following additional rules can be added to our calculus:

%\begin{defn}
%A \emph{theory} is a triple $\mathcal{T} = (G,\kappa,T)$ given by a set $G$, a regular cardinal $\kappa$ and a set $T$ of formulas in $\Formulas_\kappa(G)$. If $T$ contains a single formula $\tau \in \Formulas_\kappa(G)$, we will simply write $\mathcal{T} = (G,\kappa,\tau)$.
%\end{defn}

\begin{align*}
\prftree[r]{\scriptsize T,L} 
{S, \Gamma \Rightarrow \Delta}
{\Gamma \Rightarrow \Delta}
\end{align*}
where $S$ is a subset of $T\cup L_{(G,\kappa)}$ 
%\footnote{
%sg july25  ho messo la nota direttamente nel testo
(in view of the Cut Rule, the above additional rules are equivalent to the additional axioms
$\Rightarrow \varphi$, for $\varphi \in T\cup L_{(G,\kappa)}$).
%}

%Observe that the application of L$\bigvee$ to sequents of the form $\bigvee \emptyset, \Gamma \Rightarrow \Delta$ gives rules with no premises, which we will consider as axioms (similarly to R$\bigwedge$ applied to sequents of the form $\Gamma \Rightarrow \Delta, \bigwedge \emptyset$).

\subsection{Proofs}

The above rules are parameterized by the following four parameters: (i) a set of variables $G$; (ii)  a regular cardinal $\kappa$ giving a  bound for conjunctions and disjunctions; (iii) a logic $L$; (iv) a theory $T$.
We shall call 
$$
\mathcal{T}
= (G_\mathcal{T},\kappa_\mathcal{T}, L_\mathcal{T},T_\mathcal{T})
$$
the calculus with parameters
$G_\mathcal{T},\kappa_\mathcal{T}, L_\mathcal{T},T_\mathcal{T}$.

The formal notion of a $\mathcal{T}$-\emph{proof} (or  $\mathcal{T}$-\emph{derivation}) is the standard one: proofs are \emph{well-founded trees}\footnote{A well founded tree can be defined as a non empty 
%sg july 25 preordered set 
poset
$(\tau, \preccurlyeq)$ such that (1) for every $a \in \tau$ the set $\{b\in \tau \mid b \preccurlyeq a\} $ is finite and linearly ordered; (2) there are no strict infinite ascending chains $a_0\prec a_1\prec \cdots$ of elements of $\tau$. The latter condition allows one to  use of a strong induction reasoning principle, known as ``bar induction''.} labelled by $\Seq_{\kappa_\mathcal{T}}(G_\mathcal{T})$ (i.e.\ by sequents of formulas in $\Formulas_{\kappa_\mathcal{T}}(G_\mathcal{T})$), so that the labellings are rule applications. Since the tree underlying a proof is well founded, it is possible to assign an ordinal to a proof. A formal inductive definition is as follows (induction is up to $\lambda^+$, if $\lambda$ is the cardinality of the set of our sequents):
\begin{itemize}
 \item a sequent $\varphi\Rightarrow \varphi$ is a $\mathcal{T}$-proof (with  ordinal rank 1 and end-sequent $\varphi\Rightarrow \varphi$);
 \item if $\{\Gamma_i\Rightarrow \Delta_i\ \vert\ i \in I\}$ is a set of sequents, if
 \begin{align*}
\prftree[r]{\scriptsize }
{\cdots \Gamma_i \Rightarrow \Delta_i\cdots }
{\Gamma \Rightarrow \Delta}
\end{align*}
 is a rule of $\mathcal{T}$ and 
  $\pi_i$ are $\mathcal{T}$-proofs (with ordinals ranks $\alpha_i$ and end-sequents $\Gamma_i\Rightarrow \Delta_i$), then 
  \begin{align*}
\prftree[r]{\scriptsize }
{\cdots \pi_i\cdots }
{\Gamma \Rightarrow \Delta}
\end{align*}
is a $\mathcal{T}$-proof (with ordinal rank $(\sup_i \,\alpha_i)+1$
and end sequent $\Gamma \Rightarrow \Delta$).
\end{itemize}

For $\mathcal{T}
= (G_\mathcal{T},\kappa_\mathcal{T}, L_\mathcal{T},T_\mathcal{T})$,
we write ${\mathcal{T}} \vdash \Gamma\Rightarrow \Delta$
or $(G_\mathcal{T},\kappa_\mathcal{T}, L_\mathcal{T},T_\mathcal{T})\vdash \Gamma \Rightarrow \Delta$
(or simply 
 $\vdash \Gamma\Rightarrow \Delta$ if confusion does not arise)
 to mean that there is a $\mathcal{T}$-proof with end sequents $\Gamma\Rightarrow \Delta$.
 If $\Gamma=\emptyset$ and $\Delta=\{\varphi\}$, we write just
 ${\mathcal{T}} \vdash \varphi$
 instead of ${\mathcal{T}} \vdash ~\Rightarrow \{\varphi\}$.
 
We now show how to discharge theory axioms:

\begin{thm}[Deduction Theorem]\label{prop:ded}
Given a calculus 
$\mathcal{T}
= (G_\mathcal{T},\kappa_\mathcal{T}, L_\mathcal{T},T_\mathcal{T})$
such that $|T|<\kappa$, we have that
$$
(G_\mathcal{T},\kappa_\mathcal{T}, L_\mathcal{T},T_\mathcal{T})
%\mathcal{T} 
\vdash \Gamma \Rightarrow \Delta ~~~~~\iff~~~~~ 
(G_\mathcal{T},\kappa_\mathcal{T}, L_\mathcal{T},\emptyset)%(G,\kappa,\emptyset) 
\vdash 
%\nec^* \bigwedge T
\nec^* T, \Gamma \Rightarrow \Delta$$
for any sequent $\Gamma \Rightarrow \Delta \in \Seq_{\kappa_{\mathcal{T}}}(G_\mathcal{T})$
(here %$\nec^* \bigwedge T$
$\nec^* T$ is the set of formulas of the kind $\nec^n \varphi$, for $n\geq0$ and $\varphi \in T$).
\end{thm}
\begin{proof} An easy induction on the ordinal rank of the proof.
\end{proof}

%% file: algebras.tex
We show that our calculi are valid and complete with respect to profinite modal algebras (eventually restricted to profinite modal algebras arising in a suitable variety of modal algebras). The idea is to apply the standard Lindenbaum construction. However we need some care, because our calculi include conjunctions and disjunctions up to a fixed cardinality $\kappa$. The problem obviously concerns only completeness (not soundness) and it is solved in the following way: we show that, \emph{for sufficiently large $\kappa$, depending on the cardinality of the set $G$ of the propositional variables, our Lindenbaum algebras are indeed complete as lattices}: this dependence on the set $G$ is the major difference with respect to the customary finitary logics and is essentially justified by \cref{rem:important} from \cref{sec:pre}.

Given a profinite modal algebra $M$, let $\mathcal{U}(M)$
be its underlying set (and let $\mathcal{U}$ be the forgetful functor from profinite modal algebras to $\bf Set$).
 Given a function $v_0  \colon G \longrightarrow \mathcal{U}(M)$ (called a \emph{valuation}), we can extend it in the domain to a function $v \colon \Formulas_\kappa(G) \longrightarrow \mathcal{U}(M)$ by induction as follows.
 
 If $v_\alpha \colon \Formulas_{\kappa,\alpha}(G) \longrightarrow \mathcal{U}(M)$ has been defined, then the function 
 $$v_{\alpha+1} \colon \Formulas_{\kappa,\alpha+1}(G) \longrightarrow \mathcal{U}(M)$$ is so specified
\begin{enumerate}
    \item if $\varphi \in \Formulas_{\kappa,\alpha}(G)$, then $v_{\alpha+1}(\varphi) = v_\alpha(\varphi)$;
    \item if $\varphi = \heart \psi$ for some $\psi \in \Formulas_{\kappa,\alpha}(G)$ and $\heart \in \{\neg, \nec\}$, then
    $$v_{\alpha+1}(\varphi) = \heart v_\alpha(\psi);$$
    \item if $\varphi = \heart S$ for some $S \subseteq \Formulas_{\kappa,\alpha}(G)$ such that $|S|<\kappa$ and $\heart \in \{\bigwedge, \bigvee\}$, then
    $$v_{\alpha+1}(\varphi) = \heart v_\alpha(S).$$
\end{enumerate}
If $\beta$ is a limit ordinal, then $v_\beta \coloneqq \bigcup_{\alpha<\beta} v_\alpha \colon \Formulas_{\kappa,\beta}(G) \longrightarrow \mathcal{U}(M)$. Globally, we have $v_\kappa \colon \Formulas_\kappa(G) \longrightarrow \mathcal{U}(M)$ (we  write $v_\kappa$ simply as $v$). We can extend valuations to sequents, by setting $v(\Gamma \Rightarrow \Delta) \coloneqq v(\bigwedge \Gamma \rightarrow \bigvee \Delta)$.
%matteo rev
It is now possible to give the notion of logical consequence based on the algebraic semantics provided by profinite modal algebras.
%%%%%%%%%%%%
\begin{defn}
Given $\mathcal{T} = (G,\kappa, L,T)$ and a sequent $\Gamma \Rightarrow \Delta \in \Seq_\kappa(G)$, we write $\mathcal{T} \vDash \Gamma \Rightarrow \Delta$ to say that, for any profinite $L$-algebra $M$ and any valuation $v \colon G \longrightarrow \mathcal{U}(M)$,
\begin{align}\label{eq:sem}
\text{if } v(\varphi) = \top \text{ for each } \varphi \in T, ~~~~~ \text{then } v(\Gamma \Rightarrow \Delta) = \top\tag{\text{Sem}}
\end{align}
\end{defn}
Observe that $v(\Gamma \Rightarrow \Delta) = \top$ if and only if $\bigwedge v(\Gamma) \leq \bigvee v(\Delta)$.

\begin{prop}[Soundness]\label{prop:sound}
Given $\mathcal{T} = (G,\kappa,L,T)$ and a sequent $\Gamma \Rightarrow \Delta \in \Seq_\kappa(G)$,
$$\text{if } \mathcal{T} \vdash \Gamma \Rightarrow \Delta, ~~~~~\text{then } \mathcal{T} \vDash \Gamma \Rightarrow \Delta.$$
\end{prop}
\begin{proof}
By transfinite induction on the rank of the proof, using \cref{eq:id}, \cref{eq:ca} and \cref{eq:lf}.
\end{proof}

We now want to prove completeness. Given 
$\mathcal{T} = (G,\kappa,L,T)$
%a theory $\mathcal{T} = (G,\kappa,T)$ 
and formulas $\varphi, \psi \in \Formulas_\kappa(G)$, we set
$$\varphi \sim_\mathcal{T} \psi ~~~~~\iff~~~~~ \mathcal{T} \vdash \varphi \leftrightarrow \psi$$
The relation $\sim_\mathcal{T}$ is obviously an equivalence relation over the set $\Formulas_\kappa(G)$ (transitivity can be proved using the Cut Rule). Observe that $\varphi \sim_\mathcal{T} \psi$ is equivalent to the fact that both $\mathcal{T} \vdash \varphi \Rightarrow \psi$ and $\mathcal{T} \vdash \psi \Rightarrow \varphi$ hold. We can then introduce the \emph{Lindenbaum-Tarski algebra} associated with $\mathcal{T}$
$$\LT_L(G,\kappa,T) \coloneqq \Formulas_\kappa(G)/\sim_{\mathcal{T}}$$
$[\varphi]_\mathcal{T}$ will denote the $\sim_\mathcal{T}$-equivalence class of $\varphi \in \Formulas_\kappa(G)$. Sup's and inf's (on subsets of cardinality less than $\kappa$) and the $\nec$ operator are all defined on representatives of equivalence classes.  Observe in particular that $[\varphi]_\mathcal{T} \leq [\psi]_\mathcal{T}$ (defined  as usual as $x \leq y$ if and only if $x \wedge y = x$)
holds if and only if $\mathcal{T} \vdash \varphi \Rightarrow \psi$.
\begin{prop}\label{prop:kcompl}
$\LT_L(G,\kappa,T)$ is a $\kappa$-complete $L$-algebra.
\end{prop}
\begin{proof}
%The modal algebra operations are defined on representatives of equivalence classes.  
%
%as follows: we set, for $\varphi, \psi \in \Formulas_\kappa(G)$,
%\begin{enumerate}
%    \item $0 \coloneqq [\bigvee \emptyset]_\mathcal{T}$, $1 \coloneqq [\bigwedge \emptyset]_\mathcal{T}$;
%    \item $[\varphi]_\mathcal{T} \star [\psi]_\mathcal{T} \coloneqq [\varphi \star \psi]_\mathcal{T}$, for $\star \in \{\wedge,\vee\}$;
%    \item $\heart [\varphi]_\mathcal{T} \coloneqq [\heart \varphi]_\mathcal{T}$, for $\heart \in \{\neg,\nec\}$.
%\end{enumerate}
%It is straightforward to verify that these operations are well defined and that they make $\LT(G,\kappa,T)$ into a modal algebra.
%

%As a consequence, 
Take $X \subseteq \LT_L(G,\kappa,T)$ such that $|X|<k$, say $X = \{[\psi]_\mathcal{T}\ \vert\ \psi \in S\}$ for some $S \subseteq \Formulas_\kappa(G)$ such that $|S|<k$; then 
we have $\bigwedge X = [\bigwedge S]_\mathcal{T}$, since
%$$\mathcal{T} \vdash \bigwedge S \Rightarrow \psi \text{ for each } \psi \in S$$
%and, 
for any formula $\varphi$,
$$\mathcal{T} \vdash \varphi \Rightarrow \bigwedge S ~~~~~\iff~~~~~ \mathcal{T} \vdash \varphi \Rightarrow \psi \text{ for each } \psi \in S$$
With a similar argument, we obtain that $\bigvee X = [\bigvee S]_\mathcal{T}$.

The fact that $\LT_L(G,\kappa,T)$ is a modal algebra and that it satisfies the axiom schemata in $L$ and the formulae in $T$ follows from the rules Nec and T,L.
%L_{(G,\kappa)}$.
\end{proof}

We want to prove that, for $\kappa$ big enough, $\LT_L(G,\kappa,T)$ is a profinite $L$-algebra; for this purpose,
we fix $\mathcal{T} = (G,\kappa,L,T)$  till the end of this section and we denote simply by $\sim$ the equivalence relation 
 $\sim_\mathcal{T}$.
%
%we prove the following proposition. We will denote by $\sim$ the equivalence relation $\sim_{(G,\kappa,\emptyset)}$ for any fixed $G$ and $\kappa$. Observe that, for any theory $\mathcal{T} = (G,\kappa,T)$, we have $\sim \subseteq \sim_\mathcal{T}$.
\begin{prop}\label{prop:eq}
Let $S %matteo 17/11/2025, S_n
\subseteq \Formulas_\kappa(G)$ and $\mathscr{S} \subseteq \mathcal{P}(\Formulas_\kappa(G))$. The following equivalences hold whenever the formulas on both sides of the equivalence symbol belong to $\Formulas_\kappa(G)$:
\begin{enumerate}
    \item[{\rm (o)}] $\bigwedge \left\{ \bigwedge S \ \middle\vert\ S \in \mathscr{S}\right\}  \sim \bigwedge \left\{ \psi \ \middle\vert\ \psi \in \bigcup\mathscr{S}\right\}$
    \item[{\rm (i)}] $\neg \bigwedge S \sim \bigvee \neg S$ and $\neg \bigvee S \sim \bigwedge \neg S$;
    \item[{\rm (ii)}] $\nec \bigwedge S \sim \bigwedge \nec S$;
    \item[{\rm (iii)}] $\bigwedge \left\{\bigvee S\ \middle\vert\ S \in \mathscr{S} \right\} \sim \bigvee \left\{\bigwedge \{c(S)\ \vert\ S \in \mathscr{S}\}\ \middle\vert\ c \in \Ch(\mathscr{S})\right\}$;
    %matteo 17/11/2025 \item[\rm (iv)] $\bigwedge \left\{\nec^n \bigvee S_n\ \vert\ n \in \omega\right\} \sim \bigvee \left\{ \bigwedge \left\{ \nec^n \bigvee (I \cap S_n)\ \vert\ n \in \omega\right\}\ \middle\vert\ I \in \mathcal{P}_\omega\left(\bigcup S_n\right) \right\}$;
    \item[{\rm (iv)}] $\nec \bigvee S \sim \bigvee \left\{\nec \bigvee I\ \middle\vert\ I \in \mathcal{P}_\omega(S)\right\}$.
\end{enumerate}
\end{prop}
\begin{proof} For brevity, below we indicate directly with Ax the application of an instance of Ax followed by a weakening rule. Moreover, $S$, $c$ and $I$ will denote respectively a generic element of $\mathscr{S}$, $\Ch(\mathscr{S})$ and
%matteo 17/11/2025 $\mathcal{P}_\omega(\bigcup S_n)$;
$\mathcal{P}_\omega(S)$;
$R$, $d$ and $J$ will denote a particular choice of them. The proofs of (o), (i) and (ii) are straightforward.

For (iii), the left-to-right direction follows from the application of L$\bigwedge$, R$\bigvee$ and L$\bigvee$, resulting in branches indexed by $d \in \Ch(\mathscr{S})$. Then, for each branch
$$\{d(S)\ \vert\ S \in \mathscr{S}\} \Rightarrow \left\{\bigwedge \{c(S)\ \vert\ S \in \mathscr{S}\}\ \middle\vert\ c \in \Ch(\mathscr{S})\right\}$$
we can apply the Structural Rule W and choose, from the set on the right hand side of the sequent, the formula $\bigwedge \{d(S)\ \vert\ S \in \mathscr{S}\}$. The proof is then obtained by applying R$\bigwedge$. The other direction follows from
\begin{align*}
\prftree[r]{\scriptsize
\text{$\begin{matrix}
    \text{L}\!\bigvee \\
    (d \in \Ch(\mathscr{S}))
\end{matrix}$}}
{\cdots}
{\prftree[r]{\scriptsize
\text{$\begin{matrix}
    \text{R}\!\bigwedge \\
    (R \in \mathscr{S})
\end{matrix}$}}
{\cdots}
{\prftree[r]{\scriptsize L\text{$\bigwedge$},R\text{$\bigvee$}}
{\prftree[r]{\scriptsize Ax}
{}
{\{d(S)\ \vert\ S \in \mathscr{S}\} \Rightarrow R}}
{\bigwedge \{d(S)\ \vert\ S \in \mathscr{S}\} \Rightarrow \bigvee R}}
{\prftree[noline]{}
{\cdots}}
{\bigwedge \{d(S)\ \vert\ S \in \mathscr{S}\} \Rightarrow \bigwedge \left\{\bigvee S\ \middle\vert\ S \in \mathscr{S} \right\}}}
{\prftree[noline]{}
{\cdots}}
{\bigvee \left\{\bigwedge \{c(S)\ \vert\ S \in \mathscr{S}\}\ \middle\vert\ c \in \Ch(\mathscr{S})\right\} \Rightarrow \bigwedge \left\{\bigvee S\ \middle\vert\ S \in \mathscr{S} \right\}}
\end{align*}
since $d(R) \in R$ for each $d \in \Ch(\mathscr{S})$ and $R \in \mathscr{S}$.

%matteo 17/11/2025 (v) follows from (iv), by setting
The right-to-left implication of (iv) follows from an application of L$\bigvee$, because, for each $J \in \mathcal{P}_\omega\left(S\right)$, the sequent
$$\nec \bigvee J \Rightarrow \nec \bigvee S$$
is provable (since $J \subseteq S$). For the other direction, we apply the same trick as in \cref{rem:1}: set
\[ S_n \coloneqq
  \begin{cases}
    S       & \quad \text{if } n = 1\\
    S \cup \left\{\neg \bigvee S\right\}  & \quad \text{if } n \neq 1
  \end{cases}
\]
%matteo 17/11/2025
It is straightforward to verify that, for $n \neq 1$, $\bigvee S_n \sim \top$. As a consequence, $\nec \bigvee S \sim \bigwedge_{n \in \omega} \left\{ \nec^n \bigvee S_n \right\} \wedge \nec \bigvee S$. We can then equivalently prove
$$\bigwedge_{n \in \omega} \left\{ \nec^n \bigvee S_n \right\} \wedge \nec \bigvee S \Rightarrow \bigvee \left\{\nec \bigvee I\ \middle\vert\ I \in \mathcal{P}_\omega(S)\right\}.$$
%
%matteo 17/11/2025 For (iv), the left-to-right direction
Now the result
follows from the application of L$\bigwedge$, R$\bigvee$ and lf, resulting in branches indexed by $J \in \mathcal{P}_\omega\left(\bigcup S_n\right) %matteo 17/11/2025
= \mathcal{P}_\omega\left(S \cup \left\{\neg \bigvee S\right\}\right)
$. For each branch
%matteo 17/11/2025 $$\left\{ \nec^n \bigvee (J \cap S_n)\ \middle\vert\ n \in \omega\right\} \Rightarrow \left\{\bigwedge \left\{ \nec^n \bigvee (I \cap S_n)\ \middle\vert\ n \in \omega\right\}\ \middle\vert\ I \in \mathcal{P}_\omega\left(\bigcup S_n\right)\right\}$$
$$\left\{ \nec^n \bigvee J\ \middle\vert\ n \in \omega\right\}, \nec \bigvee S \Rightarrow \left\{\nec \bigvee I\ \middle\vert\ I \in \mathcal{P}_\omega\left(S\right)\right\}$$
we can apply the Structural Rule W and choose,
%matteo 17/11/2025
from the set on the left hand side of the sequent, the formulas $\nec \bigvee J, \nec \bigvee S$ and,
from the set on the right hand side of the sequent, the formula
%matteo 17/11/2025 $\bigwedge \left\{ \nec^n \bigvee (J \cap S_n)\ \middle\vert\ n \in \omega\right\}$.
$\nec \bigvee (J \cap S)$:
%%silvio Nov25
$$
\nec \bigvee J, \nec \bigvee S \Rightarrow \nec \bigvee (J \cap S)~.
$$
%
%matteo 17/11/2025 The proof is then obtained by applying R$\bigwedge$.
Applying the Rules Nec, R$\bigvee$ and L$\bigvee$ to $\bigvee J$, we obtain branches 
$\varphi, \bigvee S \Rightarrow J \cap S$, 
indexed by $\varphi \in J$. If $\varphi \in S$, then $\varphi, \bigvee S \Rightarrow J \cap S$ is an axiom. Otherwise, $\varphi$ must be the formula $\neg \bigvee S$, in which case an application of L$\neg$ gives the desired proof.
\end{proof}

We can now prove that each formula has a normal form and that the set of normal forms is \say{small}.
\begin{defn}\label{def:normal}
The set of \emph{normal} formulas $\Norm(G)$ over the variables $G$ consists of the formulas of the form
$$\bigwedge \left\{\bigvee S\ \middle\vert\ S \in \mathscr{S} \right\}$$
for some family $\mathscr{S} \subseteq \mathcal{P}(\Formulas_\omega(G))$. %matteo rev of sets of finitary formulas.
In other words, normal formulas are conjunctions of disjunctions of finitary formulas
\end{defn}
Observe that the cardinality of conjunctions and disjunctions involved in~\cref{def:normal} above can be at most $2^{|\Formulas_\omega(G)|}$, hence $\Norm(G)$ is a subset of $\Formulas_\kappa(G)$, for any $\kappa > 2^{|\Formulas_\omega(G)|}$.

\begin{thm}[Normal form]\label{thm:normal}
If $\kappa > 2^{|\Formulas_\omega(G)|}$, then each $\varphi \in \Formulas_\kappa(G)$ is equivalent to some normal formula, i.e.\ there is $\mathscr{S} \subseteq \mathcal{P}(\Formulas_\omega(G))$ such that
$$\varphi \sim \bigwedge \left\{\bigvee S\ \middle\vert\ S \in \mathscr{S} \right\}.$$
\end{thm}
\begin{proof}
By transfinite induction on the construction of $\varphi$. For $\varphi \in \Formulas_{\kappa,0}(G) = G \subseteq \Formulas_\omega(G)$ the claim holds with $\mathscr{S} \coloneqq \{\{\varphi\}\}$.

If $\varphi \in \Formulas_{\kappa,\alpha+1}(G)$, then we have the following cases.
\begin{enumerate}
    \item $\varphi \in \Formulas_{\kappa,\alpha}(G)$ satisfies the claim, by the induction hypothesis.
    \item $\varphi = \neg \psi$ for some $\psi \in \Formulas_{\kappa,\alpha}(G)$. We have that $\psi \sim \bigwedge \left\{\bigvee S\ \middle\vert\ S \in \mathscr{S} \right\}$ for some $\mathscr{S} \subseteq \mathcal{P}(\Formulas_\omega(G))$, by induction hypothesis. As a consequence,
    \begin{align*}
        \varphi &\sim \neg \bigwedge \left\{\bigvee S\ \middle\vert\ S \in \mathscr{S} \right\} \sim \\
        &\sim \neg \bigvee \left\{\bigwedge \{c(S)\ \vert\ S \in \mathscr{S}\}\ \middle\vert\ c \in \Ch(\mathscr{S})\right\} \sim \\
        &\sim \bigwedge \left\{\neg \bigwedge \{c(S)\ \vert\ S \in \mathscr{S}\}\ \middle\vert\ c \in \Ch(\mathscr{S})\right\} \sim \\
        &\sim \bigwedge \left\{\bigvee \{\neg c(S)\ \vert\ S \in \mathscr{S}\}\ \middle\vert\ c \in \Ch(\mathscr{S})\right\}
    \end{align*}
    by \cref{prop:eq} (i) and (iii). The claim holds for $\varphi$, since $c(S) \in S \subseteq \Formulas_\omega(G)$, hence $\neg c(S) \in \Formulas_\omega(G)$ for each $c \in \Ch(\mathscr{S})$.
    \item $\varphi = \nec \psi$ for some $\psi \in \Formulas_{\kappa,\alpha}(G)$. Similarly to the previous case,
    \begin{align*}
        \varphi &\sim \nec \bigwedge \left\{\bigvee S\ \middle\vert\ S \in \mathscr{S} \right\} \sim \bigwedge \left\{\nec \bigvee S\ \middle\vert\ S \in \mathscr{S} \right\} \sim \\
        &\sim \bigwedge \left\{\bigvee \left\{\nec \bigvee I\ \middle\vert\ I \in \mathcal{P}_\omega(S)\right\} \middle\vert\ S \in \mathscr{S} \right\}
    \end{align*}
    by \cref{prop:eq} (ii) and (v). The claim holds for $\varphi$, since $S \subseteq \Formulas_\omega(G)$, hence $\nec \bigvee I \in \Formulas_\omega(G)$ for each $S \in \mathscr{S}$ and $I \in \mathcal{P}_\omega(S)$.
    \item $\varphi = \bigwedge R$ for some $R \subseteq \Formulas_{\kappa,\alpha}(G)$ such that $|R|<\kappa$. By induction hypothesis, for each $\psi \in R$, $\psi \sim \bigwedge \left\{\bigvee S\ \middle\vert\ S \in \mathscr{S}_\psi \right\}$ for some $\mathscr{S}_\psi \subseteq \mathcal{P}(\Formulas_\omega(G))$. 
    Then
    $$
    \varphi \sim  \bigwedge\left\{
      \bigwedge \left\{ 
          \bigvee S \ \vert \ S\in \mathscr{S}_\psi 
          \right\} 
        \ \middle \vert \ \psi\in R
    \right\}
    $$
    and by \cref{prop:eq} (o)
    $$
    \varphi \sim  \bigwedge\left\{ \bigvee S
        \ \middle \vert \ S\in \bigcup \left\{ \mathscr{S}_\psi \ \vert \  \psi\in R \right\}
    \right\}
    $$
    yieldying the claim because $\bigcup \left\{ \mathscr{S}_\psi \ \vert \  \psi\in R \right\}\subseteq \mathcal{P}(\Formulas_\omega(G))$.   %
    \item $\varphi = \bigvee R$ for some $R \subseteq \Formulas_{\kappa,\alpha}(G)$ such that $|R|<\kappa$. Then $\varphi \sim \neg \bigwedge \neg R$, hence the claim follows from the application of the previous points.
\end{enumerate}

If $\varphi \in \Formulas_{\kappa,\beta}(G)$ and $\beta$ is a limit ordinal, then $\varphi \in \Formulas_{\kappa,\alpha}(G)$ for some $\alpha < \beta$, hence the claim follows from induction hypothesis.
\end{proof}

\begin{thm}\label{thm:pro}
If 
%$\kappa > 2^{2^{2^{|\Formulas_\omega(G)|}}}$,
$\kappa_0 \geq 2^{2^{|\Formulas_\omega(G)|}}$ and $\kappa> 2^{\kappa_0}$, 
then $\LT_L(G,\kappa,T)$ is a profinite $L$-algebra.
\end{thm}
\begin{proof}
The fact that $\LT_L(G,\kappa,T)$ is a \emph{complete} modal algebra follows from \cref{prop:kcompl} observing that
$$|\LT_L(G,\kappa,T)| \leq |\Norm(G)| \leq 2^{2^{|\Formulas_\omega(G)|}} \leq \kappa_0 < \kappa$$
(by \cref{thm:normal}). The claim is established by showing that 
all the identities of \cref{lem:caba,lem:lf} 
follow from~\cref{prop:eq} and from Rule lf (for the case~\eqref{eq:lf}). 
In fact, take for instance the case of the identity~\eqref{eq:id}: the two members of such identity concern families of subsets of $\LT_L(G,\kappa,T)$ and, since the elements of  $\LT_L(G,\kappa,T)$ can be represented by  normal forms, both members of~\eqref{eq:id} can be written as equivalence classes of formulas in $\Formulas_\kappa(G)$, so that~\cref{prop:eq} (iii) applies.          
A similar argument proves the identities~\eqref{eq:ca} and~\eqref{eq:lf}.
\end{proof}

We can now prove completeness for our algebraic semantics. Let us call a regular cardinal $\kappa$ \emph{sufficiently large for a set $G$} (or just \emph{sufficiently large} when reference to $G$ is understood) if it satisfies the hypotheses of~\cref{thm:pro};\footnote{ 
%sg aggiunta nota per evitare oscurità di presentazione
This happens when $\kappa > 2^{2^{2^{|\Formulas_\omega(G)|}}}$.
}
thus~\cref{thm:pro} can be restated just by saying that \emph{$\LT_L(G,\kappa,T)$ is a profinite $L$-algebra for sufficiently large $\kappa$}.

Given $\mathcal{T} = (G,\kappa,L,T)$, 
%(where $\kappa$ is sufficiently large), 
there is a canonical valuation 
$$\eta_\mathcal{T} \colon G \longrightarrow \mathcal{U}(\LT_L(G,\kappa,T))$$ 
sending $x \in G$ to $[x]_\mathcal{T} \in \LT_L(G,\kappa,T)$; observe that $\eta_\mathcal{T}(\varphi) = [\varphi]_\mathcal{T}$ for each $\varphi \in \Formulas_\kappa(G)$. In particular, if $\varphi \in T$ or if $\varphi$ is a substitution instance of a formula in $L$, then
$\eta_\mathcal{T}(\varphi) = [\varphi]_\mathcal{T} = %\left[\bigwedge \emptyset\right]_\mathcal{T} = 
\top.$

\begin{thm}[Completeness]\label{thm:completeness}
For any $\mathcal{T} = (G,\kappa,L,T)$ 
(with sufficiently large $\kappa$) and any sequent $\Gamma \Rightarrow \Delta \in \Seq_\kappa(G)$,
$$\text{if } \mathcal{T} \vDash \Gamma \Rightarrow \Delta, ~~~~~\text{then } \mathcal{T} \vdash \Gamma \Rightarrow \Delta.$$
\end{thm}
\begin{proof}
If $\mathcal{T} \vDash \Gamma \Rightarrow \Delta$, then (by \cref{eq:sem}), this sequent evaluates to $\top$ for any valuation $v$  in a profinite $L$-algebra such that $v(\varphi)=\top$ for all $\varphi \in T$. Applying this to the canonical valuation 
 $\eta_\mathcal{T}$, we get 
$$\left[\bigwedge \Gamma \right]_\mathcal{T} = \bigwedge \eta_\mathcal{T} \left(\Gamma\right) \leq \bigvee \eta_\mathcal{T} \left(\Delta\right) = \left[\bigvee \Delta \right]_\mathcal{T}$$
which
%, as observed in the proof of \cref{prop:kcompl}, 
is equivalent to $\mathcal{T} \vdash \bigwedge \Gamma \Rightarrow \bigvee \Delta$. 
\end{proof}

The following Proposition is immediate:

\begin{prop}\label{prop:univ}
Consider $\mathcal{T} = (G,\kappa,L,T)$ and  a profinite $L$-algebra $M$. Valuations $v \colon G \longrightarrow \mathcal{U}(M)$ such that $v(\varphi) = \top$ for each $\varphi \in T$ are in one to one correspondence with morphisms of $L$-algebras $h : \LT_L(G,\kappa,T) \longrightarrow M$ preserving meets and joins of cardinality less than $\kappa$.
The correspondence is given by the relationship $v=\mathcal{U}(h) \circ \eta_\mathcal{T}$.
\end{prop}

\begin{comment}
\begin{proof}
Consider a valuation $v \colon G \longrightarrow \mathcal{U}(M)$ as above. The association $[\varphi]_\mathcal{T} \mapsto v(\varphi)$ is well defined ($\varphi \sim_\mathcal{T} \psi$ implies $v(\varphi) = v(\psi)$ by soundness \cref{prop:soundness}) and defines a morphism of modal algebras $\LT_L(G,\kappa,T) \longrightarrow M$ preserving meets and joins of cardinality less than $\kappa$ (by definition of the extension of a valuation) whose set-composition with $\eta_\mathcal{T}$ gives $v$. Moreover, for any such morphism $h \colon \LT_L(G,\kappa,T) \longrightarrow M$, the identity $h([\varphi]_\mathcal{T}) = v(\varphi)$ holds.
\end{proof}

A consequence of \cref{prop:univ} is the following: given two theories $\mathcal{T} = (G,\kappa,T)$ and $\mathcal{T}' = (G,\kappa',T)$ with $\kappa, \kappa' > 2^{2^{|\Formulas_\omega(G)|}}$ (so that the corresponding Lindenbaum-Tarski algebras are profinite, by \cref{thm:pro}, hence $\kappa$-complete for each $\kappa$), we get
$$\LT_L(G,\kappa,T) \longrightarrow \LT_L(G,\kappa',T) ~~~~~(\text{associated with } \eta_{\mathcal{T}'} \colon G \longrightarrow \LT_L(G,\kappa',T))$$
and
$$\LT_L(G,\kappa',T) \longrightarrow \LT_L(G,\kappa,T) ~~~~~(\text{associated with } \eta_\mathcal{T} \colon G \longrightarrow \LT_L(G,\kappa,T))$$
which are one inverse to the other.
\end{comment} 

A consequence of the above proposition is that, if 
$\kappa, \kappa'$ are both sufficiently large,
%bigger than $2^{2^{|\Formulas_\omega(G)|}}$, 
the algebras 
$\LT_L(G,\kappa',T)$ and $\LT_L(G,\kappa,T)$ are isomorphic. This phenomenon follows from the existence of normal forms, guaranteed by the property \cref{eq:lf}, which is expressed by the rule lf: without such a rule, the Lindenbaum construction would be sensitive to the choice of the cardinal $\kappa$. From now on, we write $\mathcal{T} = (G,L,T)$ to denote the calculus $(G,\kappa, L,T)$, with $\kappa$ sufficiently large;
%big enough ($> 2^{2^{|\Formulas_\omega(G)|}}$);
similarly, we write $\LT_L(G,T)$ and $\Formulas(G)$ to denote the Lindenbaum algebra $\LT_L(G,\kappa,T)$ and the set of formulas $\Formulas_\kappa(G)$.

%% file: bridge.tex
We want to establish a connection between logical and categorical properties. As in \cite{Ghi-Zaw}, we  take into consideration \emph{r-regularity} and Barr \emph{r-exactness}, two categorical notions obtained from the extensively studied notions of regular and Barr-exact categories \say{by replacing monomorphisms with regular monomorphisms and regular epimorphisms with epimorphisms}.
%silvio 8/11/25 rimossa la frase che segue e riportata prima della dim del th 11
%In case all subobjects are regular, the two notions coincide. 
As we said at the end of the previous section, we will no longer mention the regular cardinal $\kappa$ used to build calculi acting as presentations of algebras (as we saw, $\kappa$ should be taken to be sufficiently large). Moreover, we will use the symbols $\vdash_L$ and $\vDash_L$, fixing the logical parameter $L$ of a calculus $(G,L,T)$ and letting the others vary, and write $(G,T) \vdash_L \Gamma \Rightarrow \Delta$ instead of $(G,L,T) \vdash \Gamma \Rightarrow \Delta$ (similarly for $\vDash_L$).

Many facts known from the customary finitary calculi extend (and sometimes also simplify) to our infinitary context. In particular, every profinite $L$-algebra $M$ can be presented as $\LT_L(G,T)$ for some $G$ and for some set $T \subseteq \Formulas(G)$; actually, $T$ can be taken to be a singleton $\tau \in \Formulas(G)$, because we have infinite conjunctions in our language. We avoid  parentheses for singletons, so that we write just $M\simeq \LT_L(G,\tau)$. We say that $(G,\tau)$ is a \emph{presentation} of $M$: the standard presentation is obtained by taking $G=\mathcal{U}(M)$ and 
$$\tau \coloneqq \bigwedge \{\varphi \in \Norm(G)\ \vert\ \id(\varphi) = \top\}$$
where $\id$  denote the canonical extension $\Formulas(G) \longrightarrow \mathcal{U}(M)$ of the identity valuation
(the fact that we get a presentation for $M$, i.e.\ an isomorphism $M\simeq \LT_L(G,\tau)$ can be esatablished using 
\cref{prop:univ} and \cref{thm:normal}).

\begin{comment}
; obviously, $\id(\tau) = \top$. By \cref{prop:univ}, there exists a unique morphism of profinite $L$-algebras $h \colon \LT_L(G,\tau) \longrightarrow M$, such that $h([\varphi]_{(G,\tau)}) = \id(\varphi)$ for each $\varphi \in \Formulas(G)$. Surjectivity of $h$ follows from the fact that, for each $x \in \mathcal{U}(M) = G$,
$$h\left([x]_{(G,\tau)}\right) = \id(x) = x.$$
For injectivity, consider $\varphi \in \Formulas(G)$ (wlog $\varphi \in \Norm(G)$, by \cref{thm:normal}) such that
$$\id(\varphi) = h\left([\varphi]_{(G,\tau)}\right) = \top.$$
Using the Cut rule, we get $(G,\tau) \vdash_L \varphi$ (by definition of $\tau$), hence $[\varphi]_{(G,\tau)} = \top$. As a consequence, $h$ is an isomorphism of profinite $L$-algebras.
\end{comment}

Morphisms of profinite $L$-algebras $h \colon M \longrightarrow N$ can be presented as \emph{extensions of theories}. First, present $M$ as $\LT_L(X_0,\tau_0)$, so that $h$ is isomorphic to $h_0 \colon \LT_L(X_0,\tau_0) \longrightarrow N$. Let $X_1 \coloneqq \mathcal{U}(N) \setminus \Image h_0$, with inclusion $\iota \colon X_1 \longrightarrow \mathcal{U}(N)$.
Let
$v_0 \colon X_0 \longrightarrow \mathcal{U}(N)$ be the valuation corresponding  to $h_0$ (as stated in \cref{prop:univ}); this evaluation can be extended to the 
%The universal property of the $\Set$-coproduct induces the 
valuation $v_1 \colon X_0 + X_1 \longrightarrow \mathcal{U}(N)$
taking coproduct with the inclusion $\iota$. 
Similarly as before, define
$$\tau_1 \coloneqq \bigwedge \{\varphi \in \Norm(X_0+X_1)\ \vert\ v_1(\varphi) = \top\}$$
We get an isomorphism of profinite $L$-algebras $h_1 \colon \LT_L(X_0+X_1,\tau_1) \longrightarrow N$ such that $h_1([\varphi]_{(X_0+X_1,L,\tau_1)}) = v_1(\varphi)$ for each $\varphi \in \Formulas(X_0+X_1)$. 
Making the necessary computations, it turns out that the composed morphism 
$$
\LT_L(X_0, \tau_0) \simeq M \buildrel{h}\over \longrightarrow 
N\simeq \LT_L(X_0+X_1,\tau_1)
$$
(called the standard presentation of $h$) maps $[\psi]_{(X_0,L,\tau_0)}$ to $[\psi]_{(X_0+X_1,L,\tau_1)}$, for every $\psi\in\Formulas(X_0)$. Since $\top=[\tau_0]_{(X_0,L,\tau_0)}$, this means in particular that 
$\top=[\tau_0]_{(X_0+X_1,L,\tau_1)}$, i.e.\ 
$$(X_0+X_1,\tau_1) \vdash_L \tau_0$$
meaning that \emph{$h$ represents the  extension $\tau_1$ of theory $\tau_0$ in the larger language $X_0+X_1\supseteq X_0$}.
\begin{comment}
Consider the valuation $v \colon X_0 \longrightarrow \mathcal{U}(\LT(X_0+X_1,L,\tau_1))$ whose extension sends $\psi \in \Formulas(X_0)$ to $[\psi]_{(X_0+X_1,L,\tau_1)}$. We have that $\Formulas(X_0) \subseteq \Formulas(X_0+X_1)$ and $v_1(\tau_0) = v_0(\tau_0) = \top$; as a consequence, $(X_0+X_1,\tau_1) \vdash_L \tau_0$ (we can always assume $\tau_0$ to be a normal formula, by \cref{thm:normal}), hence
$$v(\tau_0) = [\tau_0]_{(X_0+X_1,\tau_1)} = \top.$$
It is then induced a morphism of profinite $L$-algebras $\LT_L(X_0,\tau_0) \longrightarrow \LT_L(X_0+X_1,\tau_1)$ sending $[\psi]_{(X_0,\tau_0)}$ to $[\psi]_{(X_0+X_1,\tau_1)}$ for each $\psi \in \Formulas(X_0)$. We end up with the following commutative diagram in $\Pro L\MA_\fin$
\[\begin{tikzcd}[ampersand replacement=\&]
	{\LT_L(X_0,\tau_0)} \& N \\
	{\LT_L(X_0+X_1,\tau_1)}
	\arrow["{h_0}", from=1-1, to=1-2]
	\arrow[from=1-1, to=2-1]
	\arrow["{h_1}"', from=2-1, to=1-2]
\end{tikzcd}\]
with $h_1$ being an isomorphism. We then have that $h$ is isomorphic to the canonical morphism
$$\LT_L(X_0,\tau_0) \longrightarrow \LT_L(X_0 + X_1,\tau_1)$$
This means that $h$ is represented by an extension of theories, since
$$(X_0+X_1,\tau_1) \vdash_L \tau_0$$
\end{comment}
Observe that $h$ is
\begin{enumerate}
    \item injective if and only if the extension is \emph{conservative}, i.e.\
    $$(X_0+X_1,\tau_1) \vdash_L \psi ~~~~~\implies~~~~~ (X_0, \tau_0) \vdash_L \psi$$
    for each $\psi \in \Formulas(X_0)$;%\footnote{The notion of conservative extension can be equivalently formulated using (normal) formulas instead of sequents.}
    \item surjective if and only if $X_1 = \emptyset$.
\end{enumerate}
%matteo rev
We point out that the injectivity of $\LT_L(X,\tau) \longrightarrow \LT_L(Y,\sigma)$, induced by some translation $f \colon X \longrightarrow \Formulas(Y)$, corresponds to the fact that $(Y,\sigma)$ is a conservative extension of $(X,\tau)$, in the sense that
$$(Y,\sigma) \vdash_L f(\psi) ~~~~~\implies~~~~~ (X,\tau) \vdash_L \psi,$$
for each $\psi \in \Formulas(X)$, independently of the chosen presentations. By contrast, the characterization of surjectivity holds only for the specific presentations described above.
%%%%%%%%%%

We describe some colimits presentations  in $\Pro L\MA_\fin$. The initial object of $\Pro L\MA_\fin$ is $\LT_L(\emptyset,\top)$. 
A pushout square in $\Pro L\MA_\fin$ can be presented as
\[\begin{tikzcd}[ampersand replacement=\&]
	{\LT_L(X_0 + X_2,\tau_2)} \& {\LT_L(X_0 + X_1 + X_2,\tau_1 \wedge \tau_2)} \\
	{\LT_L(X_0,\tau_0)} \& {\LT_L(X_0 + X_1,\tau_1)}
	\arrow[from=1-1, to=1-2]
	\arrow[from=2-1, to=1-1]
	\arrow[from=2-1, to=2-2]
	\arrow[from=2-2, to=1-2]
\end{tikzcd}\]
where
$$(X_0+X_i,\tau_i) \vdash_L \tau_0$$
for $i = 1, 2$ (we get a coproduct taking $X_0 = \emptyset$ and $\tau_0 = \top$). Finally, coequalizers in $\Pro L\MA_\fin$ can be described for instance as
\[\begin{tikzcd}[ampersand replacement=\&]
	{\LT_L(Y,\sigma)} \& {\LT_L(X,\tau)} \& {\LT_L(X,\tau \wedge (h_0 \leftrightarrow h_1))}
	\arrow["h_0", shift left, from=1-1, to=1-2]
	\arrow["h_1"', shift right, from=1-1, to=1-2]
	\arrow[from=1-2, to=1-3]
\end{tikzcd}\]
where $h_0 \leftrightarrow h_1$ is the formula
$$\bigwedge\{\varphi_0^y \leftrightarrow \varphi_1^y\ \vert\ y \in Y\} \in \Formulas(X).$$
and $\varphi_i^y$ is a formula in $\Formulas(X)$ such that $h_i([y]_{(Y,L,\sigma)}) = [\varphi_i^y]_{(X,L,\tau)}$ for $i = 0, 1$.

\subsection{Regular and coregular decompositions}\label{subsec:beth}
%matteo rev
We begin by recalling the so-called \textit{regular} and \textit{coregular decompositions}. Let $\mathsf{C}$ be a category with kernel pairs and coequalizers of kernel pairs. Then every morphism admits a canonical decomposition as follows. Given $f \colon A \longrightarrow B$ in $\mathsf{C}$, let $k_0, k_1 \colon K \rightrightarrows A$ be its kernel pair and let $e \colon A \longrightarrow Q$ be the coequalizer of $k_0$ and $k_1$. By the universal property of the coequalizer, there exists a unique $m \colon Q \longrightarrow B$ such that $m \circ e = f$. This is called the \emph{regular decomposition}. Its dual, defined under the dual assumptions, is called the \emph{coregular decomposition}. 

%matteo rev
Each of these decompositions determines a ‘left’ and a ‘right’ class of morphisms: for example, the regular decomposition has $\mathcal{E}$ = all coequalizers of kernel pairs as ‘left’ class and $\mathcal{M}$ = all morphisms induced by the universal property of the coequalizer of some kernel pair as ‘right’ class. However, without further assumptions, $\mathcal{E}$ and $\mathcal{M}$ do not have good orthogonality properties with respect to each other. When they do, they form a so-called \textit{factorization system}.
\begin{defn}\label{defn:fact}
Given a category $\mathsf{C}$, a pair of classes of arrows $\langle \mathcal{E},\mathcal{M} \rangle$ is said to be a \emph{factorization system} for $\mathsf{C}$ if and only if the following three conditions are satisfied:% (see [FK], but we follow the equivalent formulation of [CJKP]):
\begin{enumerate}
    \item[\rm (i)] both $\mathcal{E}$ and $\mathcal{M}$ contain identities and are closed under left and right composition with isomorphisms;
    \item[\rm (ii)] each map $f$ in $\mathsf{C}$ can be written as $m \circ e$ with $m \in \mathcal{M}$ and $e \in \mathcal{E}$;
    \item[\rm (iii)] whenever we have a commutative square,
    \[\begin{tikzcd}[ampersand replacement=\&]
	A \& B \\
	C \& D
	\arrow["e", from=1-1, to=1-2]
	\arrow["g"', from=1-1, to=2-1]
	\arrow["f", from=1-2, to=2-2]
	\arrow["m"', from=2-1, to=2-2]
    \end{tikzcd}\]
    with $m \in \mathcal{M}$ and $e \in \mathcal{E}$, there is a unique $h \colon B \longrightarrow C$ such that $h \circ e = g$ and $m \circ h = f$;
\end{enumerate}
\end{defn}
The decomposition in (ii) is unique in the sense that if $f = m \circ e$ can be decomposed as well as $m' \circ e'$, for $m' \in \mathcal{M}$ and $e' \in \mathcal{E}$, then using (iii), it can be shown that there is an invertible map $h$ such that $h \circ e = e'$ and $m' \circ h = m$. In a factorization system, it turns out that both $\mathcal{E}$ and $\mathcal{M}$ are closed under composition~\cite{Bor}. % [CJKP].

%matteo rev
%Factorization systems require two independent ingredients: the existence of decompositions (condition (ii)) and a orthogonality property between the two classes (condition (iii)). \blue{[Questo paragrafo in blu non mi \`e chiarissimo.] The existence of suitable limits and colimits ensures the former for the regular factorization. However, the resulting classes need not be orthogonal to each other. In contrast, if one considers the class orthogonal to the regular epimorphisms, one obtains precisely the class of monomorphisms; here orthogonality holds, but (regular epi,mono)-decompositions need not exist in general. Thus the failure occurs on opposite sides: regular factorization may lack orthogonality, whereas the (regular epi,mono) pair may lack existence of decompositions. Since a factorization system is uniquely determined by either of its two classes, the regular factorization is a factorization system precisely if it coincides with the (regular epi,mono)-decomposition. The same holds for the dual notions.}

%matteo rev
Both regular and coregular decompositions are available in $L\KFr_\lf$, since $L\KFr_\lf$ is complete and co-complete (see \cref{prop:colim} and \cref{cor:lim}). Let us consider the coregular decomposition there. Let $f \colon W \longrightarrow V$ be a p-morphism in $L\KFr_\lf$. %matteo rev
Its cokernel pair is computed as in $\Set$, since the forgetful functor $L\KFr_\lf \longrightarrow \Set$ preserves colimits (see \cref{prop:colim}). The underlying set is
$$U \coloneqq \{(v,i) \in V \times \{0,1\}\ \vert\ \text{if } i = 0, \text{ then } v \notin \Image f\}$$
which may be described as the quotient of the disjoint union $V + V$ obtained by identifying the two copies of points in the image of $f$. The inclusions $\iota_0, \iota_1 \colon V \longrightarrow U$ are obtained by composing the inclusions $V \longrightarrow V + V$ with the projection on the aforementioned quotient, i.e.\
\[
\iota_0(v) =
\begin{cases}
(v,0) \qquad \text{if } v \notin \Image f\\
(v,1) \qquad \text{if } v \in \Image f
\end{cases}
\qquad
\iota_1(v) = (v,1).
\]
The binary relation on $U$ is given by
$$(v,i) \prec (v',i') ~~~~~\iff~~~~~ v \prec v' \text{ and }
\begin{cases}
i = i'\\
\text{or}\\
i<i' \text{ and } v' \in \Image f
\end{cases}$$
With this structure, $U$ is a locally-finite Kripke frame and $\iota_0$ and $\iota_1$ are p-morphisms. In particular, $U$ satisfies the logic $L$, being a p-morphic image of the disjoint union $V + V$. One readily checks that $\iota_0$ and $\iota_1$ form the cokernel pair of $f$ in $L\KFr_\lf$. The equalizer in $L\KFr_\lf$ of $\iota_0$ and $\iota_1$ can be easily identified with the image of $f$, since $\iota_0(v) = \iota_1(v)$ if and only if $v \in \Image f$. We then have the coregular decomposition in $L\KFr_\lf$
\[\begin{tikzcd}[ampersand replacement=\&,column sep=small]
	W \&\& V \&\& U \\
	\& {\Image f}
	\arrow["f", from=1-1, to=1-3]
	\arrow["e"', from=1-1, to=2-2]
	\arrow["{\iota_1}"', shift right, from=1-3, to=1-5]
	\arrow["{\iota_0}", shift left, from=1-3, to=1-5]
	\arrow["m"', from=2-2, to=1-3]
\end{tikzcd}\]
where $m$ is the inclusion of $\Image f$ in $V$ and $e$, which sends $w \in W$ to $f(w)$, is the unique morphism in $L\KFr_\lf$ induced by the universal property of the equalizer.

As a consequence,
\begin{prop}\label{prop:injsurj}
An arrow $f$ in $L\KFr_\lf$ is
\begin{enumerate}
    \item[(e)] an epimorphism if and only if it is surjective;
    \item[(m)] a regular monomorphism if and only if it is injective.
\end{enumerate}
The coregular decomposition in $L\KFr_\lf$ is the epi/regular mono decomposition and it forms a factorization system.
%The coregular decomposition of $f$ is the (epi,regular mono)-decomposition.
\end{prop}
\begin{proof}
For (e), observe that $f$ is an epimorphism if and only if $\iota_0 = \iota_1$, if and only if $m$ is an isomorphism, if and only if $f$ is surjective.

For (m), recall that an arrow is a regular monomorphisms if and only if it equalizes its cokernel; thus $f$ is a regular monomorphism if and only if $e$ is an isomorphism, if and only if $f$ is injective.

For what we said before, if $m \circ e$ is the coregular decomposition of $f$ in $L\KFr_\lf$, then $m$ is a regular monomorphism and $e$ is an epimorphism; it is easy to check that regular monomorphisms and epimorphisms are orthogonal, in the sense of (iii) of \cref{defn:fact}. By \cref{prop:colim} and \cref{cor:lim}, the coregular decomposition exists for every morphism in $L\KFr_\lf$: this gives the last claim (condition (ii) of \cref{defn:fact} is easy to check).
\end{proof}

If we consider the dual category $\Pro L\MA_\fin$, we get that the regular decomposition of a morphism $h \colon M \longrightarrow N$ in $\Pro L\MA_\fin$ is preserved by the forgetful functor and hence it is based on the set-theoretic image of $h$ (this comes from a general fact for categories monadic over $\bf Set$~\cite{EV}). The relation between regular and coregular decomposition in the category $\Pro L\MA_\fin$ can be expressed in a purely syntactical way. Let $\mathcal{T} = (X+Y,L, \tau)$ be a theory, and let $\varphi(X,Y) \in \Formulas(X+Y)$. The \emph{Strong Beth's Definability Property}\footnote{This is called (in an equivalent formulation for finitary languages) \emph{projective} Beth property in~\cite{Mak99}.
The idea of investigating Beth property via regular and coregular decompositions is due to M. Makkai.
}  
says that whenever we have
\begin{align}\label{eq:I}
\tau(X,Y_0) \wedge \tau(X,Y_1) \vdash_L \varphi(X,Y_0) \leftrightarrow \varphi(X,Y_1)\tag{I}
\end{align}
we also have that
\begin{align}\label{eq:E}
\text{there exists } \psi(X) \in \Formulas(X), \text{ s.t.\ } \tau(X,Y)\vdash_L \varphi(X,Y) \leftrightarrow \psi(X)\tag{E}
\end{align}
(here $Y_0$ and $Y_1$ are disjoint copies of $Y$).

\begin{thm}\label{thm:BethStrong}
Strong Beth's Definability Property holds for $\vdash_L$ if and only if the regular and the coregular decompositions coincide in $\Pro L\MA_\fin$.
\end{thm}
\begin{proof}
Consider a morphism of profinite modal algebras $M \longrightarrow N$, which can be presented as $h \colon \LT_L(X,\sigma) \longrightarrow \LT_L(X+Y,\tau)$. Its cokernel pair in $\Pro L\MA_\fin$ is given by
\[\begin{tikzcd}[ampersand replacement=\&]
	{\LT_L(X+Y,\tau)} \& {\LT_L(X+Y_0+Y_1, \tau(X,Y_0) \wedge \tau(X,Y_1))}
	\arrow["q_0", shift left, from=1-1, to=1-2]
	\arrow["q_1"', shift right, from=1-1, to=1-2]
\end{tikzcd}\]
The equalizer of $q_0$ and $q_1$ is the subalgebra of $\LT_L(X+Y,\tau)$ of the equivalence classes of formulae $\varphi(X,Y)$
%elements $[\varphi]_{(X+Y,T)}$
such that \cref{eq:I} holds (this is the coregular decomposition); it always contains the image of $h$, which is the subalgebra of the equivalence classes of the formulae $\varphi(X,Y)$ such that \cref{eq:E} holds (as we pointed out before, the set-theoretic image gives the regular decomposition in $\Pro L\MA_\fin$). The two sets coincide precisely when Strong Beth Definability Property happens to be true.
\end{proof}

We can rephrase the characterization of~\cref{thm:BethStrong} as follows:
\begin{thm}\label{thm:BethStrong1}
Strong Beth's Definability Property holds for $\vdash_L$ if and only if all monomorphisms are regular in $\Pro L\MA_\fin$
(dually, if and only if all epimorphisms are regular in $L\KFr_\lf$).
\end{thm}
 \begin{proof}
    From~\cref{prop:injsurj}, we know that the coregular decomposition in $L\KFr_\lf$ is the epi/regular mono decomposition. Thus, if all epis are regular, such a decomposition coincides with the regular one. Vice versa, if the two decompositions coincide, all epis are regular (because the second component in the coregular decomposition of an epi is the identity).
 \end{proof}

In $\Pro L\MA_\fin$, if the regular and the coregular decompositions coincide, not only monos are regular, but epis are regular too (by the reverse argument). However, regularity of epis and monos are not equivalent conditions, as we show in Example~\ref{ex:wbeth} 
below.
%(keep in mind that we have  an  epi/regular mono factorization system in $L\KFr_\lf$ - dually a regular epi/mono factorization system in $\Pro L\MA_\fin$).
 
Regularity of epis in $\Pro L\MA_\fin$, 
syntactically speaking, corresponds to another version of the Beth Definability Property. This version 
(that we call \emph{Weak Beth Definability Property} or simply \emph{Beth Definability Property}) says that 
for all $X,Y, \tau(X,Y)$, if we have 
\begin{align}\label{eq:Iw}
\tau(X,Y_0) \wedge \tau(X,Y_1) \vdash_L 
\bigwedge_{y\in Y} (y_0\leftrightarrow y_1)
\tag{wI}
\end{align}
we also have that, for each $y\in Y$
\begin{align}\label{eq:Ew}
\text{there is a }  \psi_y(X) \in \Formulas(X), \text{ s.t.\ } \tau(X,Y)\vdash_L  \bigwedge_{y\in Y} (y\leftrightarrow \psi_y(X))\tag{wE}
\end{align}
(here again $Y_0$ and $Y_1$ are disjoint copies of $Y$).
%obtained by the strong one by taking
% as the formula $\varphi(X,Y)$ the formula  $y$ itself.
\begin{prop}\label{prop:wB}
Beth's Definability Property holds for $\vdash_L$ if and only if all epimorphisms are regular in $\Pro L\MA_\fin$
%sg july25  aggiunta riga per simmetria col teorema precedente
(dually, if and only if all monomorphisms are regular in $L\KFr_\lf$).
\end{prop}
\begin{proof}
Consider the same situation as in \cref{thm:BethStrong}.
%, with $Y = \{y\}$. 
We have that $h$ is an epimorphism if and only if $q_0 = q_1$ if and only if \cref{eq:Iw} holds. On the other side, $h$ is a regular epimorphism if and only if it is surjective if and only if \cref{eq:Ew} holds.  
\end{proof}

\begin{rem}
Our formulation of the Beth Definability Property does not match with the current terminology used in  finitary languages, where it is assumed that the set $Y$ in \eqref{eq:Iw} and \eqref{eq:Ew} is a singleton (or, equivalently, that it is finite~\cite{Ghi-Zaw}). Since our formulas may contain infinitely many variables, such finitary formulation does not look natural. In the finitary version, the Beth definability Property is known to hold for all intermediate logics and for all logics above $K4$~\cite{Mak92}, because all epimorphisms between finitely presented algebras in varieties of Heyting or $K4$ algebras are regular (but same is not true if the restriction to finitely presented algebras is dropped~\cite{Mor17,Mor20}). 
\end{rem}

We show that Weak Beth Definability Property is really weaker than Strong Beth Definability Property:

\begin{example}\label{ex:wbeth}
Consider the Kripke frame $V$ having $\{0,1,2\}$ as underlying set and such that $\prec = \{(0,1),(0,2),(1,2),(2,1)\}$, as represented below.
\begin{center}
\begin{tikzpicture}[scale=.7]
  \node (one) at (-1,1) {$1$};
  \node (two) at (1,1) {$2$};
  \node (zero) at (0,-1) {$0$};
  \draw[->] (zero) -- (one);
  \draw[->] (zero) -- (two);
  \draw[->] (one) to [out=45, in=135] (two);
  \draw[->] (two) to [out=-135, in=-45] (one);
\end{tikzpicture}
\end{center}
Consider the logic $L = K +_f \Sigma$, where $\Sigma$ is the following set of modal formulas:
\begin{align*}
    &\pos \top; \\
    &\nec (\pos x \rightarrow \nec x); \\
    &\pos \pos x \to (x \vee \pos x); \\
    &\pos x \to \nec (\neg x \to \pos x); \\
    &\pos x \wedge \pos y \to \nec (x \vee y).
    %&\pos x \wedge \pos (y \wedge \neg x) \rightarrow \nec [ (x \vee y) \wedge (x \rightarrow \nec y) \wedge (y \rightarrow \nec x))]
\end{align*}
%matteo rev
%In \cref{sec:pre}, we mentioned Kripke semantics and its connection with algebraic semantics provided by the Goldblatt-Thomason's duality (\cref{thm:duality}). One way to easily identify the class of Kripke frames satisfying a given set of axioms is to translate propositional modal formulas into first order formulas in via the so-called \emph{standard translation}: a first order variable $w$ is fixed; propositional variables become unary relation symbols
The axioms of $L$ express the following first-order conditions on Kripke frames: for each $w$, there is $v$ such that $w \prec v$; for each $w$ and $v$, if $w \prec v$, then there is at most one point $u$ such that $v \prec u$; for each $w$, $v$ and $u$, if $w \prec v$ and $v \prec u$, then either $w = u$, or $w \prec u$; for each $w$ and $v$, if $w \prec v$, then any other point $u$ different from $v$, with $w \prec u$, is such that $u \prec v$; for each $w$ there are at most two points $v$ and $u$ with $w \prec v$ and $w \prec u$.\footnote{The conditions of this and of the following example can be obtained with the software SQEMA \cite{MR2295764}.} Then,
$$L\KFr_\lf = \{W \in \KFr\ \vert\ \forall w \in W(w^* \text{ is a subreduction\footnotemark of } V)\}.$$\footnotetext{A subreduction of $V$ is a p-morphic image of some generated subframe of $V$ (a generated subframe of $V$ is obtained by taking a $\prec^*$-closed subset of $V$ and by restricting to it the relation $\prec$).}
We can list and represent the subreductions of $V$, other than $V$ itself (up to isomorphism):
\begin{itemize}
    \item $V_1$ having $\{1,2\}$ as underlying set and such that $\prec = \{(1,2),(2,1)\}$;
    \item $V_2$ having $\{0,1\}$ as underlying set and such that $\prec = \{(0,1),(1,1)\}$;
    \item the one-element reflexive Kripke frame $\mathbf{1}$.
\end{itemize}
\begin{center}
\begin{tikzpicture}[scale=.7]
    \node (one) at (-1,0) {$1$};
    \node (two) at (1,0) {$2$};
    \draw[->] (one) to [out=45, in=135] (two);
    \draw[->] (two) to [out=-135, in=-45] (one);
\begin{scope}[shift={(4,0)}]
\node (one) at (0,1) {$1$};
    \node (zero) at (0,-1) {$0$};
    \draw[->] (zero) -- (one);
    \draw[->] (0.2,1) arc (-90:200:0.4);
\begin{scope}[shift={(4,0)}]
    \node (one) at (0,0) {$1$};
    \draw[->] (0.2,0) arc (-90:200:0.4);
\end{scope}
\end{scope}
\end{tikzpicture}
\end{center}

On the one hand, the unique map $V\longrightarrow \bf 1$ is an epimorphism, being surjective (see \cref{prop:injsurj}), but not a regular epimorphism. To see this, consider the kernel pair $k_0, k_1 \colon K \longrightarrow V$ of $V\longrightarrow \bf 1$ (it is the category-theoretic product of $V$ with itself, $\bf 1$ being terminal in $L\KFr_\lf$). By $K \in L\KFr_\lf$, for any $u \in K$, the generated subframe $u^*$ must be a subreduction of $V$. Moreover, restricting $k_0$ and $k_1$ to $u^*$, we obtain p-morphisms $u^* \longrightarrow V$. These two conditions together (being a subreduction of $V$ and being the domain of a p-morphism with codomain $V$) force $u^*$ to be isomorphic to either $V$ or $V_1$ and $k_0$ and $k_1$ to be (isomorphic to) either the inclusion or the map swapping $1$ and $2$. It is now clear that the coequalizer of $k_0$ and $k_1$ is the surjective map $q \colon V\longrightarrow V_2$, which is not isomorphic to $V \longrightarrow \bf 1$. This means that Strong Beth's Property does not hold (see \cref{thm:BethStrong1}).
 
 On the other hand, we show that all monomorphisms are injective in $L\KFr_\lf$ (recall that Weak Beth's Property holds if and only if all monomorphisms are regular in $L\KFr_\lf$, by \cref{prop:wB}, and that the regular monomorphisms are the injective p-morphisms in $L\KFr_\lf$, by \cref{prop:injsurj}). First notice, by case inspection, that any
 %matteo rev1 root-preserving (i.e. mapping a root to a root)
 surjective
 monomorphism among two subreductions of $V$
 %matteo rev1 frames  in $L\KFr_\lf$
 is an isomorphism. Pick now an arbitrary monomorphism $f: W\longrightarrow U$ and $w_1\neq w_2\in W$ such that $f(w_1)=f(w_2):=u \in U$.
 %matteo rev1 We cannot have $w_1\in w_2^*$ or $w_2\in w_2^*$, by the case inspection above. Again, by the same case inspection, the restrictions $w_i^* \longrightarrow f(w^*_i) = u^*$ of the p-morphism $f$ (for $i=1,2$) must be isomorphisms.
 The restrictions $w_i^* \longrightarrow f(w^*_i) = u^*$ of the p-morphism $f$, for $i = 1,2$, are surjective monomorphisms. By the case inspection above, they must be isomorphisms.
 It is now easy to define two distinct parallel p-morphisms $u^* \longrightarrow W$ whose composite with $f$ are equal, contradicting the fact that $f$ is a monomorphism.
\end{example}

In general, (weak, hence strong) Beth's Definability Property does not hold, as shown by the following counterexample taken from~\cite{Copr}. 

\begin{example}\label{ex:beth}
Consider the Kripke frame $V$ having $\{0,1\}$ as underlying set and such that $\prec = \{(0,0), (0,1), (1,0)\}$, as represented below.
\begin{center}
\begin{tikzpicture}[scale=.7]
    \node (zero) at (-1,0) {$0$};
    \node (one) at (1,0) {$1$};
    \draw[->] (zero) to [out=45, in=135] (one);
    \draw[->] (one) to [out=-135, in=-45] (zero);
    \draw[->] (-0.9,0.25) arc (-20:270:0.4);
\end{tikzpicture}
\end{center}
Consider the logic $L = K +_f \Sigma$, where $\Sigma$ is the following set of modal formulas:
\begin{align*}
    &\pos \top; \\ % (x \vee \neg x) \\
    %matteo rev &\pos \pos x \rightarrow (x \vee \pos x) \\
    &x \to \nec \pos x; \\
    &\pos x \to \pos \pos x; \\
    &(x \wedge \pos (\neg x \wedge y)) \to \nec (x \vee y); \\
    &(x \wedge \pos y) \to (\pos \nec x \vee \nec y).
    %&(x \wedge \pos y) \to \nec (x \vee y); \\
    %&(x \wedge \pos y) \to \pos \pos (\nec x \vee \nec y).
    %&x \to (\pos x \leftrightarrow \pos \nec x) \\
    %&x \wedge \pos (y \wedge \neg x) \rightarrow \nec (\neg x \rightarrow y)
    %matteo rev &x \wedge \pos \neg x \rightarrow (\pos x \leftrightarrow \pos \nec x)
\end{align*}
The axioms of $L$ express the following first-order conditions on Kripke frames: for each $w$, there is $v$ such that $w \prec v$; for each $w$ and $v$, if $w \prec v$, then $v \prec w$; for each $w$ and $v$, if $w \prec v$, then there exists $u$ such that $w \prec u$ and $u \prec v$; for each $w$ and $v$, if $w \prec v$ and $w \neq v$, then any point $u$ such that $w \prec u$ is either equal to $w$ or to $v$; for each $w$ and $v$, if $w \prec v$, then either $v$ is the unique point above $w$, or there exists $v_0$ such that $w \prec v_0$ and for which $v_0 \prec u$ implies $u = w$ for all $u$.
%for each $w$ and $v$, if $w \prec v$, then any point $u$ such that $w \prec u$ is either equal to $w$ or to $v$; for each $w$ and $v$, if $w \prec v$, then we can find $u$ and $u'$, with $w \prec u$ and $u \prec u'$ such that either no point above $u'$ is different from $w$, or no point above $u'$ is different from $v$.
Then,
%\begin{equation}\label{eq:V}
$$L\KFr_\lf = \{W \in \KFr_\lf\ \vert\ \forall w \in W (w^* \text{ is a p-morphic image of } V)\}.$$%\tag{V}
%\end{equation}
In fact, up to isomorphism, the p-morphic images of $V$ are just $V$ itself and the one-element reflexive Kripke frame $\bf 1$. Moreover, any p-morphism among two p-morphic images of $V$ must be the identity.
%matteo rev Now notice that for every $W\in L\KFr_\lf$ and $w\in W$, any morphism $w^*\longrightarrow V$ must be the identity, because $w^*$ is a p-morphic image of $V$.
%, according to \eqref{eq:V}.
%Hence given 
As a consequence, given
a pair of parallel p-morphisms $f,g \colon W \longrightarrow V$ in $L\KFr_\lf$, it must be $f=g$.
%
%the surjective p-morphism $V \longrightarrow w^*$ %in the characterization of $L\KFr_\lf$
%equalizes $f|_{w^*}$ and $g|_{w^*}$ for each $w \in W$ (notice that any p-morphism with domain $V$ must be either the identity or the one-point collapsing p-morphism);
%$V \longrightarrow V$ must be the identity); 
%as a consequence, $f = g$. 
%
This implies that any morphism in $L\KFr_\lf$ having $V$ as domain is a monomorphism. In particular, the p-morphism $V \longrightarrow \mathbf{1}$
%matteo rev1 , where $\mathbf{1}$ is the one-point reflexive Kripke frame,
is a monomorphism in $L\KFr_\lf$ that is not regular (regular monomorphisms are injective p-morphisms, by \cref{prop:injsurj}).
\end{example}

%matteo rev
%Let us consider the logic $K4$. Recall that...
%
However, Weak Beth's Definability Property holds for extensions of $K4$.
\begin{prop}\label{prop:regmono}
If $L$ extends $K4$, then all monomorphisms are regular in $L\KFr_\lf$ (thus all epimorphisms are regular in $\Pro L\MA_\fin$). 
\end{prop}
\begin{proof}
Recall that the assumption of $L$ extending $K4$ means that the Kripke models of $L$ are transitive (see the Table of \cref{table:logics}).

Let $f \colon W \longrightarrow V$ be a monomorphism in $L\KFr_\lf$.

Suppose we have $w \in W$ such that there exists $w' \in W$ with $w' \neq w$ and $f(w') = f(w)$; consider $w$ having this property, for which $|w^*|$ is minimal (in $L\KFr_\lf$ we only have locally-finite frames, hence the finite cardinality $|w^*|$ of $w^*$ is well defined).

Consider first the case where $w' \in w^*$; by the minimality condition with respect to the property above, it must be that $w \in w'^*$ (otherwise $|w'^*| < |w^*|$), hence $w^*=w'^*$. Thanks to transitivity, the function $\sigma_{w,w'} \colon w^* \longrightarrow w^*$ swapping $w$ and $w'$ (and fixing all the other points) defines a p-morphism (hence a morphism in $L\KFr_\lf$, since $w^*$ is a generated subframe of $W$). By $f(w)=f(w')$, the following diagram in $L\KFr_\lf$
\[\begin{tikzcd}
	{w^*} & W \\
	W & V
	\arrow["\iota", from=1-1, to=1-2]
	\arrow["{\iota \circ \sigma_{w,w'}}"', from=1-1, to=2-1]
	\arrow["f", from=1-2, to=2-2]
	\arrow["f"', from=2-1, to=2-2]
\end{tikzcd}\]
commutes (the morphism $\iota$ is the inclusion $w^* \subseteq W$). The morphism $f$ being a monomorphism, we have that $\iota \circ \sigma_{w,w'} = \iota$; in particular, $w = w'$, which is a contradiction.

Notice that we also proved that
%We have proved that 
the restriction $w^* \longrightarrow f(w^*) = f(w)^*$ of the p-morphism $f$ (it defines a morphism in $L\KFr_\lf$) is injective, hence it is an isomorphism (it is surjective by definition and it is straightforward to see that bijective p-morphisms are isomorphisms in $\KFr$); we can then consider its inverse $h \colon f(w^*) \longrightarrow w^*$ (this $h$ will be useful below).

We are left with the case $w' \not\in w^*$.
%Given $w' \neq w$ such that
Since $f(w) = f(w')$, we have $f(w^*) = f(w)^* = f(w')^* = f(w'^*)$, so we can consider the morphism $g \colon w'^* \longrightarrow w^*$ given by the composition of the restriction $w'^* \longrightarrow f(w')^* = f(w)^*$ of $f$ with $h \colon f(w^*) \longrightarrow w^*$. By definition, the following diagram in $L\KFr_\lf$
\[\begin{tikzcd}
	{w'^*} & W \\
	W & V
	\arrow["{\iota'}", from=1-1, to=1-2]
	%\arrow["{\iota g}"', from=1-1, to=2-1]
        \arrow["\iota \circ g"', from=1-1, to=2-1]
	\arrow["f", from=1-2, to=2-2]
	\arrow["f"', from=2-1, to=2-2]
\end{tikzcd}\]
commutes (the morphism $\iota'$ is the inclusion $w'^* \subseteq W$). The morphism $f$ being a monomorphism, we have that 
%$\iota g = \iota'$;
$\iota \circ g  = \iota'$;
in particular, $w = w'$, which is a contradiction.

We conclude that $f$ is injective, i.e.\ that it is a regular monomorphism (by \cref{prop:injsurj}).
\end{proof}

\subsection{r-Regularity}

We have seen in \cref{prop:injsurj} that the coregular decomposition of $L\KFr_\lf$ is the epi/regular mono decomposition and that it forms a factorization system:
%\blue{([inciso non chiaro]  equivalently, the coregular factorization forms a factorization system, as in \cref{defn:fact})}:
in this section we investigate \emph{stability} of this factorization system and relate it to a further well-known syntactic property, namely Craig's Interpolation Theorem. Let us take into consideration the definition of factorization system provided in \cref{defn:fact}.
\begin{defn}\label{defn:stab}
A factorization system $\langle \mathcal{E},\mathcal{M} \rangle$ for a category $\mathsf{C}$ is said to be \emph{stable} if $\mathsf{C}$ has finite limits and the following further condition is satisfied
 \begin{enumerate}
    \item[\rm (iv)] whenever we have a pullback square
    \[\begin{tikzcd}[ampersand replacement=\&]
	A \& B \\
	C \& D
	\arrow["{f'}", from=1-1, to=1-2]
	\arrow["{e'}"', from=1-1, to=2-1]
	\arrow["e", from=1-2, to=2-2]
	\arrow["f"', from=2-1, to=2-2]
    \end{tikzcd}\]
    the fact that $e \in \mathcal{E}$ implies that $e' \in \mathcal{E}$.
\end{enumerate}
\end{defn}

\begin{defn}
We say that a category $\mathsf{C}$ is \emph{r-regular}~\cite{Ghi-Zaw} if it has finite limits and moreover, taking all epimorphisms as $\mathcal{E}$ and all regular monomorphisms as $\mathcal{M}$, we get a stable factorization system for $\mathsf{C}$.
\end{defn}
%matteo rev
As conditions (i) and (iii) of \cref{defn:fact} are trivially true in this case (by the definition of epi and regular mono), C is r-regular if and only if it has finite limits, each arrow has an epi-regular mono decomposition and epimorphisms are stable under pullbacks.
%As we observed in the previous sections, if every morphism admits a factorization as an epimorphism followed by a regular monomorphism, then the epi/regular mono pair is a factorization system. We can therefore rephrase the previous definition in the following way: $\mathsf{C}$ is r-regular if and only if it has finite limits, each arrow has an epi/regular mono factorization and epimorphisms are stable under pullbacks.

\begin{rem}\label{rem:reg}
As we said in the introduction of \cref{sec:prop}, r-regularity is a variation of the notion of \emph{regularity}: the latter requires that $\mathcal{E}$ = all regular epimorphisms and $\mathcal{M}$ = all monomorphisms form a stable factorization system.
\end{rem}
%%%%%%%%%%%%%%%%%%

Consider now $L\KFr_\lf \simeq \Pro L\MA_\fin^\op$. By \cref{prop:injsurj}, epimorphisms are surjective p-morphisms and regular monomorphisms are injective p-morphisms and each arrow has an epi/regular mono decomposition. Moreover, as we said in \cref{cor:lim}, $L\KFr_\lf$ has all (finite) limits. Condition (iv) of \cref{defn:stab} can be decomposed into two parts: thanks to the epi/regular mono decomposition, we can check stability of epimorphisms under pullbacks separately 
for the case where the pullback is taken along a 
 regular monomorphism or along an epimorphism.
\begin{prop}\label{prop:cep}
In $L\KFr_\lf$ epimorphisms are stable under pullback along regular monomorphisms.
\end{prop}
\begin{proof}
Consider a diagram in $L\KFr_\lf$
\[\begin{tikzcd}[ampersand replacement=\&]
	\& W \\
	U \& V
	\arrow["f", from=1-2, to=2-2]
	\arrow["m"', from=2-1, to=2-2]
\end{tikzcd}\]
with $m$ injective (regular mono). Consider the Kripke frame having
$$f^*(U) \coloneqq \{w \in W\ \vert\ \exists u \in U (f(w) = m(u))\}$$
as underlying set, endowed with the restriction of the binary relation of $W$ (it belongs to $L\KFr_\lf$, being a generated subframe of $W \in L\KFr_\lf$); the inclusion $m' \colon f^*(U) \longrightarrow W$ is an injective p-morphism. Moreover, the association $f^*(U) \ni w \mapsto u \in U$ such that $f(w)=m(u)$ (such $u$ is unique) defines a p-morphism $f' \colon f^*(U) \longrightarrow U$. It is straightforward to verify that
\[\begin{tikzcd}[ampersand replacement=\&]
	{f^*(U)} \& W \\
	U \& V
	\arrow["m'", from=1-1, to=1-2]
	\arrow["f'"', from=1-1, to=2-1]
	\arrow["f", from=1-2, to=2-2]
	\arrow["m"', from=2-1, to=2-2]
\end{tikzcd}\]
is a diagram of pullback in $L\KFr_\lf$. The claim follows from the fact that $f'$ is surjective whenever $f$ is so.
\end{proof}

\begin{rem}
\cref{prop:cep} is the dual of the \emph{congruence extension property}. In universal algebra, an algebraic structure $B$ has the congruence extension property if: for any substructure $A \leq B$ and any congruence $\alpha$ on $A$, there exists a congruence $\beta$ on $B$ such that $\beta \cap (A \times A) = \alpha$. A class of algebraic structures has the congruence extension property if each of its members has the congruence extension property. Rephrasing it in categorical terms (using the one-to-one correspondence between congruences and quotients, typical of the categories of algebras), we say that a category $\mathsf{C}$ has the congruence extension property (CEP) if: given a monomorphism $k \colon A \longrightarrow B$ and a regular epimorphism $g \colon A \longrightarrow C$, there are a monomorphism $h \colon C \longrightarrow D$ and a regular epimorphism $f \colon B \longrightarrow D$ such that $f \circ k = h \circ g$. If $\mathsf{C}$ has pushouts, then (CEP) holds in $\mathsf{C}$ if and only if monomorphisms are stable under pushout along regular epimorphisms. As a consequence of \cref{prop:cep}, (CEP) holds in  $\Pro L\MA_\fin$.
\end{rem}

As a consequence of~\cref{prop:cep}, $L\KFr_\lf$ is r-regular if and only if epimorphisms are stable under pullbacks along epimorphisms. The last condition is equivalent to a nice algebraic property for the dual category $\Pro L\MA_\fin$.
\begin{defn}
A category $\mathsf{C}$ has the \emph{amalgamation property} (AP) if, for any monomorphisms $k \colon A \longrightarrow B$, $g \colon A \longrightarrow C$, there are monomorphisms $h \colon C \longrightarrow D$, $f \colon B \longrightarrow D$ such that $f \circ k = h \circ g$, i.e.\ such that the square
\[\begin{tikzcd}
	A & B \\
	C & D
	\arrow["k", from=1-1, to=1-2]
	\arrow["g"', from=1-1, to=2-1]
	\arrow["f", from=1-2, to=2-2]
	\arrow["h"', from=2-1, to=2-2]
\end{tikzcd}\]
commutes. The pair of monomorphisms $h$ and $f$ is called \emph{amalgamation} for the pair $k$ and $g$.
\end{defn}

It is straightforward to verify that, in a category $\mathsf{C}$ with pushouts, (AP) holds if and only if monomorphisms are stable under pushout along monomorphisms, yielding the following.
\begin{cor}\label{cor:regap}
$L\KFr_\lf$ is r-regular if and only if $\Pro L\MA_\fin$ has the amalgamation property.
\end{cor}

\begin{rem}\label{rem:grz}
It is not difficult to see that
$\Pro L\MA_\fin$ has the amalgamation property if and only if $L\MA_\fin$
has it, too (an easy way to see this is via duality with $L\KFr_\lf$ and $L\KFr_{fin}$). However, the fact that $L\MA_\fin$ has amalgamation \emph{is not sufficient for $L\MA$ to have it} (even if $L$ has finite model property, as always in this paper): a notable counterexample is \emph{Grz.Lin}, see~\cite{Mak82}. 
On the other hand, above $K4$, if $L\MA$ has amalgamation and $L$ has finite model property, it turns out that $L\MA_\fin$ (and consequently also $\Pro L\MA_\fin$) has amalgamation, too (see~\cite[Proposition 2.20]{Ghi-Zaw}).
\end{rem}

We now go to the syntactic side. Let $\varphi \in \Formulas(X_0+X_1)$ and $\psi \in \Formulas(X_0+X_2)$.  \emph{Craig's Theorem} says that whenever
$$\varphi(X_0,X_1) \vdash_L \psi(X_0,X_2)$$
we have that there exists $\sigma \in \Formulas(X_0)$ such that
$$\varphi(X_0,X_1) \vdash_L \sigma(X_0) ~~~~~\text{and}~~~~~ \sigma(X_0) \vdash_L \psi(X_0,X_2).$$
The formula $\sigma$ is called \emph{interpolant} for $\varphi$ and $\psi$.

\begin{thm}\label{thm:craig=amalg}
Craig's Theorem holds for $\vdash_L$ if and only if $\Pro L\MA_\fin$ has the amalgamation property, if and only if $L\KFr_\lf$ is r-regular.
\end{thm}
\begin{proof}
We prove that Craig's Theorem holds if and only if injective morphisms of profinite $L$-algebras (dual to epimorphisms in $L\KFr_\lf$) are stable under pushouts in $\Pro L\MA_\fin$ (see \cref{cor:regap}).

Assume Craig's Theorem. Consider a diagram in $\Pro L\MA_\fin$
\[\begin{tikzcd}[ampersand replacement=\&]
	{M_2} \\
	{M_0} \& {M_1}
	\arrow[from=2-1, to=1-1]
	\arrow[from=2-1, to=2-2]
\end{tikzcd}\]
with $M_0 \longrightarrow M_1$ injective; it can be rewritten in the following form
\[\begin{tikzcd}[ampersand replacement=\&]
	{\LT_L(X_0 + X_2,\tau_2)} \\
	{\LT_L(X_0,\tau_0)} \& {\LT_L(X_0 + X_1,\tau_1)}
	\arrow[from=2-1, to=1-1]
	\arrow[from=2-1, to=2-2]
\end{tikzcd}\]
with $(X_0 + X_1,\tau_1)$ conservative extension of $(X_0,\tau_0)$, as observed at the beginning of \cref{sec:prop}. The pushout in $\Pro L\MA_\fin$ of the diagram above is
\[\begin{tikzcd}[ampersand replacement=\&]
	{\LT_L(X_0 + X_2,\tau_2)} \& {\LT_L(X_0 + X_1 + X_2,\tau_1 \wedge \tau_2)} \\
	{\LT_L(X_0,\tau_0)} \& {\LT_L(X_0 + X_1,\tau_1)}
	\arrow[from=1-1, to=1-2]
	\arrow[from=2-1, to=1-1]
	\arrow[from=2-1, to=2-2]
	\arrow[from=2-2, to=1-2]
\end{tikzcd}\]
$\LT_L(X_0 + X_2,\tau_2) \longrightarrow \LT_L(X_0 + X_1 + X_2,\tau_1 \wedge \tau_2)$ is injective if and only if the corresponding extension is conservative, i.e.\ if and only if
$$\tau_1(X_0,X_1)\wedge \tau_2(X_0,X_2) \vdash_L \varphi ~~~~~\implies~~~~~ \tau_2(X_0,X_2) \vdash_L \varphi$$
for each $\varphi \in \Formulas(X_0+X_2)$.
Consider $\varphi \in \Formulas(X_0+X_2)$ satisfying the premise of the condition above; using \cref{prop:ded} and $\nec^* (\tau_1 \wedge \tau_2) \sim \nec^* \tau_1 \wedge \nec^* \tau_2$ (the equivalence follows from the fact that $\nec$ preserves meets), we have that
$$\tau_1(X_0,X_1) \vdash_L \nec^* \tau_2 \rightarrow \varphi.$$
Applying Craig's Theorem to $\tau_1 \in \Formulas(X_0+X_1)$ and $\nec^* \tau_2 \rightarrow \varphi \in \Formulas(X_0+X_2)$, we can find $\sigma \in \Formulas(X_0)$ such that
$$\tau_1(X_0,X_1) \vdash_L \sigma(X_0) ~~~~~\text{and}~~~~~ \sigma(X_0) \vdash_L \nec^* \tau_2(X_0,X_2) \rightarrow \varphi(X_0,X_2).$$
From the first entailment, using the fact that $(X_0 + X_2,\tau_1)$ is a conservative extension of $(X_0,\tau_0)$,
%matteo rev $\LT_L(X_0,\tau_0) \longrightarrow \LT_L(X_0 + X_2,\tau_1)$ is injective,
%$(X_0, \tau_0) \subseteq (X_0+X_1, \tau_1)$ is conservative (by injectivity of the corresponding morphism), 
we get $\tau_0(X_0) \vdash_L \sigma$; moreover, $\tau_2(X_0,X_2)\vdash_L \tau_0$.
Applying suitable cuts, we finally obtain
\begin{comment}
 composing $(X_0+X_2,\tau_2) \vdash_L \tau_0$, $(X_0,\tau_0) \vdash_L \sigma$ and $(X_0+X_2,\sigma) \vdash_L \nec^* \tau_2 \rightarrow \varphi$, together with
\begin{align*}
\prftree[r]{\scriptsize Cut}
{\prftree[noline]
%{\prffancysummarybox}
{\Rightarrow \nec^* \tau_2 \rightarrow \varphi}}
{\prftree[r]{\scriptsize L\text{$\rightarrow$}}
{\prftree[r]{\scriptsize Nec}
{\prftree[r]{\scriptsize Nec}
{\prftree[r]{\scriptsize \text{$\tau_2$}-Ax}
{\Rightarrow \tau_2}}
{\vdots}}
{\Rightarrow \nec^* \tau_2}}
{\prftree[r]{\scriptsize Ax}
{\varphi \Rightarrow \varphi}}
{\nec^* \tau_2 \rightarrow \varphi \Rightarrow \varphi}}
{\Rightarrow \varphi}
\end{align*}
we obtain 
\end{comment}
$\tau_2(X_0,X_2) \vdash_L \varphi$.

Vice versa, assume amalgamation property for $\Pro L\MA_\fin$, and consider $\varphi \in \Formulas(X_0+X_1)$ and $\psi \in \Formulas(X_0+X_2)$ such that
$$\varphi(X_0,X_1) \vdash_L \psi(X_0,X_2).$$
Define
$$\sigma \coloneqq \bigwedge \{\rho \in \Norm(X_0)\ \vert\ \varphi(X_0,X_1) \vdash_L \rho(X_0)\} \in \Formulas(X_0)$$
We prove that $\sigma$ is an interpolant. Obviously, $\varphi(X_0,X_1) \vdash_L \sigma$, hence the following
\[\begin{tikzcd}[ampersand replacement=\&]
	{\LT_L(X_0 + X_2,\sigma)} \& {\LT_L(X_0 + X_1 + X_2,\varphi)} \\
	{\LT_L(X_0,\sigma)} \& {\LT_L(X_0 + X_1,\varphi)}
	\arrow[from=1-1, to=1-2]
	\arrow[from=2-1, to=1-1]
	\arrow[from=2-1, to=2-2]
	\arrow[from=2-2, to=1-2]
\end{tikzcd}\]
is a pushout diagram in $\Pro L\MA_\fin$. $\LT_L(X_0,\sigma) \longrightarrow \LT_L(X_0 + X_1,\varphi)$ is injective (the corresponding extension is conservative, by the definition of $\sigma$). By amalgamation property, the morphism $\LT(X_0 + X_2,\sigma) \longrightarrow \LT(X_0 + X_1 + X_2,\varphi)$ is conservative, too. In particular, we get $\sigma(X_0) \vdash_L \psi(X_0,X_2)$.
\end{proof}

\begin{rem}
The above proof is  more simple than the proof of the corresponding result in the finitary fragment of the language
%~\cite{Mak},
because  in our infinitary language there is no difference between interpolants and \emph{uniform} interpolants and so we directly looked for a uniform interpolant in the second part of the proof.
\end{rem}

r-Regularity can be encountered in many cases. For example, r-regularity holds whenever $L\KFr_\lf$ can be characterized via \emph{universal Horn clauses} in the first order language containing the binary predicate $\prec$ and equality.
%matteo rev
A universal Horn clause is a condition of the form
\begin{align*}%\label{eq:Horn}
    A_1 \wedge \dots \wedge A_n \to C%\tag{H}
\end{align*}
where the $A_i$'s and $C$ are atomic formulas. For example, if $L = K4$, then $L\KFr_\lf$ is the class of those locally finite Kripke frames that are transitive, i.e.\ that satisfy the Horn clause $(w_1 \prec w_2) \wedge (w_2 \prec w_3) \to (w_1 \prec w_3)$. %%%%%%%%%%%%%
Under such Horn's definability assumption,  the dual of the amalgamation property (\cref{cor:regap}) is checked as follows. Consider a pair of p-morphisms
\[\begin{tikzcd}[ampersand replacement=\&]
	\& W_1 \\
	W_0 \& V
	\arrow["f_1", from=1-2, to=2-2]
	\arrow["f_0"', from=2-1, to=2-2]
\end{tikzcd}\]
in $L\KFr_\lf$. Define
$$U \coloneqq \{(w_0,w_1) \in W_0 \times W_1\ \vert\ f_0(w_0) = f_1(w_1)\}$$
and set, for $(w_0,w_1), (w'_0,w'_1) \in W$,
$$(w_0,w_1) \prec (w'_0,w'_1) ~~~~~\iff~~~~~ w_0 \prec w'_0 ~\&~ w_1 \prec w'_1.$$
With such definitions, $W$ is a locally finite Kripke frame and the projection $\pi_i \colon U \longrightarrow W_i$ is a p-morphism, which is surjective if $f_i$ is so, for $i = 0, 1$. Moreover, thanks to the fact that $L\KFr_\lf$ is characterized by Horn clauses, we have that $W \in L\KFr_\lf$.
%sg july25 aggiunta frasetta
This proves amalgamation in $L\KFr_\lf$ (but notice that $W$ is usually not the pullback in $L\KFr_\lf$).

%silvio 8/11/25
%As we observed at the beginning of \cref{sec:prop}, 
%
If all the monomorphisms are regular (i.e.\ if the Weak Beth's Definability Property holds, recall~\cref{prop:wB}), then the notions of r-regular and of regular categories coincide. This is a general fact:
% non c'è bisogno della Prop 7
%this is because in such situation epis turns out to be extremal epis.
%by \cref{prop:injsurj}.
\begin{lem}\label{lem:rreg}
Suppose that in a category  $\mathsf{C}$ monomorphisms are all regular. Then  $\mathsf{C}$ is r-regular if and only if it is regular.   
\end{lem}
\begin{proof}
Recall that a category with finite limits is regular if and only if extremal epis and monos form a stable factorization system~\cite[Section A1.3]{elephant} (extremal epis are the arrows that cannot factorize through a proper mono as a second component). Observe that, when all monos are regular, we also have that all epis are extremal. In fact, if an epi factors as $m\circ f$, then $m$ is epi (as second component of an epi); if $m$ is also mono and all monos are regular, then $m$ is iso. This proves that, under the assumption that all monos are regular, epis = extremal epis and regular monos = monos. The claim now follows by recalling that the definition of r-regular category requires epis and regular monos to form a stable factorization system.
%The right-to-left side is trivial (when all monos are regular, all epis are extremal). For the other side, it is sufficient to show that, in the current hypotheses, epis are all extremal. In fact, if an epi factors as $m\circ f$, then $m$ is epi (as second component of an epi); if $m$ is also mono and all monos are regular, then $m$ is iso: this shows that epis are extremal.
\end{proof}

The above lemma applies,
%This happens 
for instance, whenever $L$ extends $K4$, by \cref{prop:regmono}. In such situations, Strong Beth Definability Property follows from Craig's Theorem:
\begin{thm}
 If the Weak Beth's Definability Property and the Craig's Theorem both hold for $\vdash_L$, then the Strong Beth's Definability Property holds, too.   
\end{thm}
 
\begin{proof}
According to~\cref{thm:BethStrong1}, we have to show that an epi $f \colon W \longrightarrow V$ in $L\KFr_\lf$ is regular. But we are in  a regular category (by Craig's Theorem and \cref{lem:rreg}), hence regular epis coincide with extremal epis, by Joyal's Theorem~\cite[Proposition 1.3.4]{elephant}.
%, i.e.\ with those arrows that cannot factorize through a proper mono as a second component. 
In our situation, monos are precisely injective functions (see~\cref{prop:injsurj,prop:wB}), so surjectivity of $f$ (see~\cref{prop:injsurj}) is sufficient to ensure that $f$ is an extremal epi and hence a regular epi.
\end{proof}
 
Without the assumption  of Weak Beth Definability Property, however, \emph{the statement of the previous Theorem does not hold}: in fact, amalgamation property is easily seen to be true for the logic $L$ mentioned in~\cref{ex:beth}, where nevertheless Beth's Properties fail.

 \begin{comment}
 We can obtain the following characterizations for varieties of interior algebras.
\begin{prop}\label{prop:reg}
If $L$ extends $S4$, then $L\KFr_\lf$ is regular if and only if it is one of the following:
\begin{enumerate}
    \item[{\rm (I)}] $\{W \in \Pos_\lf\ \vert\ \width(W) \leq 1\}$ ($L\MA$ is the variety of $S4.3.\textit{Grz}$-algebras);
    \item[{\rm (II)}] $\{W \in \PreO_\lf\ \vert\ \height(W) \leq l ~\&~ \width(W) \leq v ~\&~ \internal(W) \leq m ~\&~ \external(W) \leq n\}$, for some $l \in \{0, 1, 2\}$ and some $v, m, n \in \{1, 2, \omega\}$;
    \item[{\rm (III)}] $\{W \in \PreO_\lf\ \vert\ \internal(W) \leq m ~\&~ \external(W) \leq n ~\&~ \confluent(W) \leq s\}$, for some $m, n \in \{1, 2, \omega\}$ and some $s \in \{1, \omega\}$.
\end{enumerate}
\end{prop}
\begin{proof}
See \cite{Copr}.
\end{proof}
\end{comment}

\subsection{Local Interpolation}

So far, we only considered \emph{global} consequence relation (i.e.\ deduction under hypotheses). Here we show an interesting characterization concerning \emph{local} consequence relation.

 Let $\varphi \in \Formulas(X_0+X_1)$ and $\psi \in \Formulas(X_0+X_2)$.  \emph{Local Craig's Theorem} says that whenever
$$\vdash_L \varphi(X_0,X_1)\to  \psi(X_0,X_2)$$
we have that there exists $\sigma \in \Formulas(X_0)$ such that
$$\vdash_L\varphi(X_0,X_1) \to \sigma(X_0) ~~~~~\text{and}~~~~~ \vdash_L \sigma(X_0) \to \psi(X_0,X_2).$$
The formula $\sigma$ is called a \emph{local interpolant} for $\varphi$ and $\psi$.

It is clear that Local Craig's Theorem implies its global version, namely Craig's Theorem (as stated in the previous Section), because of the Deduction~\cref{prop:ded}. A standard algebraic characterization of Local Craig's Theorem for modal logics is obtained by the notion of \emph{superamalgamation}.

\begin{defn}
%matteo rev
A commutative square
\[\begin{tikzcd}
	A & B \\
	C & D
	\arrow["k", from=1-1, to=1-2]
	\arrow["g"', from=1-1, to=2-1]
	\arrow["f", from=1-2, to=2-2]
	\arrow["h"', from=2-1, to=2-2]
\end{tikzcd}\]
in $\Pro L\MA_\fin$ is said to have the \emph{superamalgamation property} if: for every $c \in C$ and $b \in B$, if $h(c) \leq f(b)$, then there is $a \in A$ such that $c \leq g(a)$ and $k(a) \leq b$.

Moreover, we
%%%%%%%%%%%%% We
say that $\Pro L\MA_\fin$ has the \emph{superamalgamation property} (SAP) if, for any monomorphisms $k \colon A \longrightarrow B$, $g \colon A \longrightarrow C$, there are monomorphisms $h \colon C \longrightarrow D$, $f \colon B \longrightarrow D$ such that $f \circ k = h \circ g$
%matteo rev , i.e.\ such that the square
%\[\begin{tikzcd}
%	A & B \\
%	C & D
%	\arrow["k", from=1-1, to=1-2]
%	\arrow["g"', from=1-1, to=2-1]
%	\arrow["f", from=1-2, to=2-2]
%	\arrow["h"', from=2-1, to=2-2]
%\end{tikzcd}\]
%commutes
and such a commutative square has the superamalgamation property.
%matteo rev The latter means that for every $c\in C, b\in B$, if $h(c) \leq f(b)$, then there is $a\in A$ such that $c\leq g(a)$ and $k(a)\leq b$.
\end{defn}

The following fact can be established reasoning as in the proof of \cref{thm:craig=amalg}

\begin{prop}
Local version of Craig's Theorem holds for $\vdash_L$ if and only if $\Pro L\MA_\fin$ has the superamalgamation property.
\end{prop}

The following result gives a categorical characterization of the Local Craig's Theorem involving the forgetful functor and weak pullbacks:

\begin{thm}\label{thm:craig}
Local Craig's Theorem holds for $\vdash_L$ if and only if the forgetful functor $L\KFr_\lf \longrightarrow \Set$ preserves weak pullbacks.
\end{thm}

\begin{proof}
First, let us observe that $L$ has (SAP) if and only if all pushouts in $\Pro L\MA_\fin$ have the superamalgamation property, if and only if all weak pushouts in $\Pro L\MA_\fin$ have the superamalgamation property (use the universal properties to establish these equivalences). Thus it is sufficient to show that a commutative square of profinite algebras 
    \[\begin{tikzcd}
	A & B \\
	C & D
	\arrow["k", from=1-1, to=1-2]
	\arrow["g"', from=1-1, to=2-1]
	\arrow["f", from=1-2, to=2-2]
	\arrow["h"', from=2-1, to=2-2]
    \end{tikzcd}\]
    has the superamalgamation property if and only if the dual square 
    \[\begin{tikzcd}
	A^* & B^* \\
	C^* & D^*
	\arrow["k^*", from=1-2, to=1-1]
	\arrow["g^*"', from=2-1, to=1-1]
	\arrow["f^*", from=2-2, to=1-2]
	\arrow["h^*"', from=2-2, to=2-1]
    \end{tikzcd}\]
     of locally finite Kripke frames becomes a weak pullback once the forgetful functor is applied to it.
     This is easily established as follows. The (SAP) property for the dualized square means that, for all $c\subseteq C^*, b\subseteq B^*$, we have
     $$
     (h^*)^{-1}(c) \subseteq (f^*)^{-1}(b)~~\Rightarrow~~
     \exists a\subseteq A^*~(c\subseteq  (g^*)^{-1}(a)~\&~
      (k^*)^{-1}(a)\subseteq b)~~.
     $$
     Playing with the fact that inverse image has both a left and a right adjoint, this is the same as asking that, for all $c\subseteq C^*, b\subseteq B^*$, we have 
     $$
     \exists_{f^*}(h^*)^{-1}(c)\subseteq b 
     ~~\Rightarrow~~(k^*)^{-1}\exists_{g^*}(c)\subseteq b 
     $$
     (here $\exists_{f^*}$ and $\exists_{g^*}$ are the left adjoints). Since the latter has to hold for all $b$, it reduces to requiring 
     $$
     (k^*)^{-1}\exists_{g^*}(c) \subseteq  \exists_{f^*}(h^*)^{-1}(c)
     $$
     for all $c\subseteq C^*$. In turn, such condition is verified if and only if it is verified when $c$ is a singleton, i.e.\ if and only if
     $$
     \forall x\in C^*\, \forall y\in B^*\,(k^*(y)=g^*(x) 
     ~\Rightarrow~\exists z\in D^* (h^*(z)=x ~\&~f^*(z)=y))
     $$
     The latter  is precisely the condition for a commutative square in $\Set$ to be a weak pullback.
\end{proof}

\subsection{r-Exactness}
Another interesting characterization theorem can be established by taking into consideration the notion of Barr r-exactness (this is the notion of Barr exactness, modified for categories which are not regular, but r-regular --- see below). 

If $\mathsf{C}$ is a category with finite limits, then for every object $A$ we can take all regular monomorphisms of codomain $A$; if we consider two regular monomorphisms equal in case they differ by an $A$-isomorphism, we do get, in fact, a semilattice $\Sub_r(A)$, whose elements are called \emph{regular subobjects} of $A$. Given an arrow $f \colon B \longrightarrow A$, taking pullbacks induces a semilattice morphism $f^* \colon \Sub_r(A) \longrightarrow \Sub_r(B)$. This association is functorial. 

Take now a morphism $f \colon A \longrightarrow B$ in $\mathsf{C}$ and perform its kernel pair
\[\begin{tikzcd}[ampersand replacement=\&]
	K \& A \\
	A \& B
	\arrow["{k_0}", from=1-1, to=1-2]
	\arrow["{k_1}"', from=1-1, to=2-1]
	\arrow["f", from=1-2, to=2-2]
	\arrow["f"', from=2-1, to=2-2]
\end{tikzcd}\]
Then, the (equivalence class of the) induced $\langle k_0, k_1 \rangle \colon K \longrightarrow A \times A$ is a regular subobject of $A \times A$. In addition, $\langle k_0, k_1 \rangle$ satisfies the following properties (see~\cite[Definition 1.3.6]{elephant}):
\begin{enumerate}
    \item \emph{Reflexivity}. there exists $r \colon A \longrightarrow K$ such that $k_0 \circ r = \id_A$ and $k_1 \circ r = \id_A$;
    \item \emph{Symmetry}. there exists $s \colon K \longrightarrow K$ such that $k_0 \circ s = k_1$ and $k_1 \circ s = k_0$;
    \item \emph{Transitivity}. if the square
    \[\begin{tikzcd}[ampersand replacement=\&]
	P \& K \\
	K \& A
	\arrow["{p_1}", from=1-1, to=1-2]
	\arrow["{p_0}"', from=1-1, to=2-1]
	\arrow["{k_0}", from=1-2, to=2-2]
	\arrow["{k_1}"', from=2-1, to=2-2]
    \end{tikzcd}\]
    is a pullback, then there exists $p \colon P \longrightarrow K$ such that $k_0 \circ p = k_0 \circ p_0$ and $k_1 \circ p = k_1 \circ p_1$.
\end{enumerate}
The above properties can be synthetically reformulated by saying that $\langle k_0, k_1 \rangle$ is an internal r-equivalence relation in $\mathsf{C}$.

\begin{defn}
An \emph{r-equivalence relation} in $\mathsf{C}$ over the object $A$ is a regular subobject $\langle k_0, k_1 \rangle \colon K \longrightarrow A \times A$ satisfying the Reflexivity, Symmetry and Transitivity conditions above; r-equivalence relations of the form \say{kernel pair of some morphism} are called \emph{effective}.
\end{defn}

\begin{defn}
A category $\mathsf{C}$ is said to be \emph{Barr r-exact} if it is r-regular and all r-equivalence relations are effective.
\end{defn}

%matteo rev
\begin{rem}
Recall that \emph{Barr-exactness} differs from r-Barr-exactness only in that the regularity assumption on subobjects is omitted and regularity of the category is required instead of r-regularity (see \cref{rem:reg}). Here, by a subobject of $A$ we mean an equivalence class of monomorphisms, up to isomorphism.
\end{rem}

Consider the dual notions in $\Pro L\MA_\fin$. Take a presentation $\LT_L(X,\tau)$ of a profinite $L$-algebra $M$; then the coproduct $M + M$ can be presented as $\LT_L(X_0+X_1,\tau(X_0) \wedge \tau(X_1))$, where $X_0$ and $X_1$ are disjoint copies of $X$ and $\tau(X_i)$ denotes the formula obtained by substituting the variables $X_i$ into the variables $X$ in $\tau$. A regular epimorphism in $\Pro L\MA_\fin$, having $M + M$ as domain can then be presented as
\[\begin{tikzcd}[ampersand replacement=\&]
	{\LT_L(X_0+X_1,\tau(X_0) \wedge \tau(X_1))} \& {\LT_L(X_0+X_1,\rho)}
	\arrow[from=1-1, to=1-2]
\end{tikzcd}\]
(regular monomorphisms in $L\KFr_\lf$ are injective p-morphisms, recall~\cref{prop:injsurj}). In other words, to give a regular epimorphism in $\Pro L\MA_\fin$ having $\LT_L(X,\tau) + \LT_L(X,\tau)$ as domain, it is sufficient to give a formula $\rho = \rho(X_0,X_1)$ such that
\[\rho(X_0,X_1) \vdash_L \tau(X_0) \wedge \tau(X_1)\tag{0}\]

We can prove the following.
\begin{prop}\label{prop:coeq}
The regular epimorphism
\[\begin{tikzcd}[ampersand replacement=\&]
	{\LT_L(X_0+X_1,\tau(X_0) \wedge \tau(X_1))} \& {\LT_L(X_0+X_1,\rho)}
	\arrow[from=1-1, to=1-2]
\end{tikzcd}\]
satisfies the dual of
\begin{enumerate}
    \item[\rm (i)] Reflexivity if and only if
    \[\tau(X)\vdash_L \rho(X,X);\tag{1}\]
    \item[\rm (ii)] Symmetry if and only if
    \[\rho(X_0,X_1) \vdash_L \rho(X_1,X_0);\tag{2}\]
    \item[\rm (iii)] Transitivity if and only if
    \[\rho(X_0,X_1) \wedge \rho(X_1,X_2) \vdash_L \rho(X_0,X_2).\tag{3}\]
\end{enumerate}
\end{prop}
In other words, an r-coequivalence relation in $\Pro L\MA_\fin$ over $\LT_L(X,\tau)$ is given by a formula $\rho = \rho(X_0,X_1)$ satisfying (0)-(3).

On the other hand, the cokernel pair  of any morphism
\[\begin{tikzcd}[ampersand replacement=\&]
	{\LT_L(Y,\sigma)} \& {\LT_L(X,\tau)}
	\arrow["h", from=1-1, to=1-2]
\end{tikzcd}\]
in $\Pro L\MA_\fin$ can be presented as
\[\begin{tikzcd}[ampersand replacement=\&]
	{\LT_L(X,\tau)} \& {\LT_L(X_0+X_1,\rho)}
	\arrow[shift left, from=1-1, to=1-2]
	\arrow[shift right, from=1-1, to=1-2]
\end{tikzcd}\]
with $\rho(X_0,X_1)$ (again, $X_0$ and $X_1$ are disjoint copies of $X$) having the form
\begin{align*}\label{eq:sv}
\tau(X_0) \wedge \tau(X_1) \wedge \bigwedge\{\varphi^y(X_0) \leftrightarrow \varphi^y(X_1)\ \vert\ y \in Y\}\tag{SV}
\end{align*}
where $\{\varphi^y\}$ is a family of formulas in $\Formulas(X)$ indexed by the context $Y$ of the algebra %\blue{Ma  non va bene  indiciare le $\varphi$ semplicemente su $X$, visto che $LT(Y,\sigma$) \`e stato sostituita da $LT(X,\tau)$? Non c'era nemmeno bisogno di cambiare la presentazione del dominio di $h$, solo di quella del codominio.}
$\LT(Y,\sigma)$, such that $h([y]_{(Y,L,\sigma)}) = [\varphi^y]_{(X,L,\tau)}$ for each $y \in Y$.

\begin{defn}
Consider a formula $\tau(X)$ and let $X_0$ and $X_1$ be disjoint copies of $X$. We say that a formula $\rho(X_0,X_1)$ \emph{separates variables} %$X_0$ \emph{and} $X_1$
\emph{relative to} $\tau(X)$ if it is equivalent to the formula of form \cref{eq:sv} for some formulas $\{\varphi^y\}_{y \in Y} \subseteq \Formulas(X)$.\footnote{Here $Y$ plays the role of a generic indexing set. In the proof of the next result it will become a context.}
\end{defn}

We can prove the following.
\begin{prop}
A formula $\rho(X_0,X_1)$ defines an effective r-coequivalence relation on $\LT_L(X,\tau)$ if and only if $\rho(X_0,X_1)$ separates variables %$X_0$ and $X_1$
relative to $\tau(X)$.
\end{prop}
\begin{proof}
If $\rho(X_0,X_1)$ defines an effective r-coequivalence relation on $\LT_L(X,\tau)$ then, as we observed above, $\rho(X_0,X_1)$ separates variables. On the other hand, if there are formulas $\{\varphi^y\}_{y \in Y} \subseteq \Formulas(X)$ such that $\rho$ is (equivalent to)
$$\tau(X_0) \wedge \tau(X_1) \wedge \bigwedge\{\varphi^y(X_0) \leftrightarrow \varphi^y(X_1)\ \vert\ y \in Y\}$$
then the pair of morphisms
\[\begin{tikzcd}[ampersand replacement=\&]
	{\LT_L(X,\tau)} \& {\LT_L(X_0+X_1,\rho)}
	\arrow[shift left, from=1-1, to=1-2]
	\arrow[shift right, from=1-1, to=1-2]
\end{tikzcd}\]
is the cokernel pair in $\Pro L\MA_\fin$ of the morphism $\LT_L(Y,\top) \longrightarrow \LT_L(X,\tau)$ associated to the valuation $Y \longrightarrow \mathcal{U}(\LT_L(X,\tau))$ sending $y \in Y$ to the equivalence class $[\varphi^y]_{(X,L,\tau)}$.
\end{proof}

\begin{cor}
$L\KFr_\lf$ is Barr r-exact if and only if all formulas defining r-coequivalence relations separate variables for $\vdash_L$.
\end{cor}

Clearly, when all monos are regular in $L\KFr_\lf$ (i.e.\ when Weak Beth's Definability Property holds) Barr r-exactness becomes the same as Barr exactness. Barr exactness is an extremely strong property for $L\KFr_\lf$, because the remaining conditions for Giraud's Theorem~\cite[Theorem 1.4.5]{makkaireyes} are satisfied.
%matteo rev
Indeed, the subcategory $L\KFr_\fin$ provide a small generating set, and the forgetful functor $L\KFr_\lf \longrightarrow \Set$ preserves colimits (see \cref{prop:colim}) and pullbacks along injective p-morphisms (see the proof of \cref{prop:cep}), which implies that small coproducts are disjoint and stable under pullback. Consequently, if $L\KFr_\lf$ is Barr exact, then it is a Grothendieck topos.
%%%%%%%
Thus one expects that Barr exactness should not hold in categories like $L\KFr_\lf$. We nevertheless found two examples:

\begin{example}\label{ex:gllin}
Consider  $\GLLin_\lf$, namely the category  of  locally finite, transitive and irreflexive Kripke frames  $(W,\prec)$ such that  the restriction of the relation $\prec$ to each rooted generated subframe is  linear  (these are the Kripke frames for the logic \emph{GL.Lin} mentioned in the Table of~\cref{sec:pre}). $\GLLin_\lf$ is Barr exact because there exists an equivalence of categories between $\GLLin_\lf$ and the category of presheaves $\Set^{(\omega,\leq)^\op}$, see \cite{Copr}.
%There exists some axiomatization $L$ extending $K4$ (hence the notion of r-Barr exactness coincide with the standard one) for which $L\KFr_\lf = \GLLin_\lf$ (using splitting formulas, see \cite[Chapter 9.4]{Modal}). %Observe that the restriction of a p-morphism $f \colon W \longrightarrow V$ in $\GLLin_\lf$ to any $w^*$ must be injective: if $w' \neq w'' \in w^*$ are such that $v \coloneqq f(w') = f(w'')$, then, since one between $w' \prec w''$ and $w'' \prec w'$ must hold, we would have $v \prec v$, which is not possible by irreflexivity. This means that $|w^*| = |f(w^*)| = |f(w)^*|$.
\end{example}

\begin{example}
An example where transitivity does not hold, is given by the category $D\KFr_\lf$ of locally finite Kripke frames $(W,\prec)$ where  the relation $\prec$ is a function (these are just sets endowed with an ultimately periodic unary function: $D\KFr_\lf$ can then be seen as a full subcategory of the category of presheaves $\Set^{\mathbb{N}^\op}$ over the monoid $\mathbb{N}$). They can be axiomatized  by  the axiom
$$\pos x = \nec x.$$
It is straightforward to verify that the forgetful functor $D\KFr_\lf \longrightarrow \Set$ preserves not only colimits (see~\cref{prop:colim}), but also finite limits: Barr exactness follows easily.
%matteo rev
In fact, $D\KFr_\lf$ is equivalent to the topos of continuous actions of the profinite completion of $\mathbb{N}$ as a monoid \cite{MR4535663}. Recall that, if $A$ is an algebra in a variety $\mathbf{V}$, the \emph{profinite completion} of $A$ as a $\mathbf{V}$-algebra is defined as the limit --- taken in the category of topological $\mathbf{V}$-algebras --- of the filtered diagram of quotients $A \longrightarrow C$, where $C$ ranges over finite $\mathbf{V}$-algebras and it is endowed with the discrete topology.
%not only colimits preserves (finite) limits. As a consequence, we have that $D\KFr_\lf$ is r-regular (epimorphisms are surjective p-morphisms and pullbacks are preserved by $U$) and regular (all monomorphisms are injective, since pullbacks are preserved by $U$). Moreover, it is possible to prove that $D\KFr_\lf$ is Barr-exact (using that $U$ preserves both pullbacks and coequalizers --- see \cref{prop:colim} --- and that it reflects isomorphisms).
\end{example}

\begin{comment}
On the other hand, in~\cite{Copr} it is shown that there are exacly five $L$ extending $\bf S4$;

and such that 
$L\KFr_\lf$ is Barr exact: these comprises, besides the logics axiomatized by $\bot$ and $\Box x =x$, only three other uninteresting varieties where $L\KFr_\lf$ is equivalent to the topos of the action of a finite small monoid.

a full classification is giving for the extensions $L$ of $\bf S4$ such that 
$L\KFr_\lf$ is Barr exact: these comp

As we did for regularity, restricting ourself to the case in which $L$ extends $S4$, we obtain the following nice characterization.
\begin{prop}\label{prop:exact}
If $L$ extends $S4$, then $L\KFr_\lf$ is Barr exact if and only if it is one of the following:
\begin{enumerate}
    \item[{\rm 0.}] $\{\emptyset\}$;
    \item[{\rm 1.}] $\{W \in \Pos_\lf\ \vert\ \height(W) \leq 1\}$;
    \item[{\rm 2.}] $\{W \in \Pos_\lf\ \vert\ \height(W) \leq 2 ~\&~ \width(W) \leq 1\}$;
    \item[{\rm 3.}] $\{W \in \PreO_\lf\ \vert\ \height(W) \leq 1 ~\&~ \external(W) \leq 2\}$;
    \item[{\rm 4.}] $\{W \in \PreO_\lf\ \vert\ \height(W) \leq 2 ~\&~ \width(W) \leq 1 ~\&~ \internal(W) \leq 2 ~\&~ \external(W) \leq 1\}$.
\end{enumerate}
\end{prop}
\begin{proof}
See \cite{Copr}.
\end{proof}

\end{comment}

%% file: conclusions.tex
In this paper we introduced an infinitary calculus, whose algebraic models (= Lindenbaum algebras of theories inside such calculus) are precisely profinite modal algebras. We investigated, for a normal modal propositional logic $L$ with the finite model property, the correspondences between syntactic properties of this calculus enriched with the axioms of $L$ and categorical properties of the corresponding class of profinite $L$-algebras.  
We consider the framework presented in this paper as a starting point for a broader research line, some of whose directions are outlined below.

\begin{enumerate}
    \item[{\rm 1.}] From a strictly syntactic point of view, it would be nice to supply a cut-free version of the calculus presented in~\cref{sec:calc}. A semantic cut elimination proof has been given in~\cite[Theorem 22.17]{Tak} for the non-modal fragment; an extension of such result to suitable reformulations of our calculi for basic logics like $K, K4,S4$, etc. would be desirable. Recently, there has been progress in syntactic cut elimination proofs for countably infinitary versions of classical and intuitionistic logics~\cite{Tesi}, although satisfactory solutions for modal  extensions are still to be investigated~\cite{TesiModal,minari}.  It should be noted, however, that our context is different from the one considered in the above mentioned papers, because on the one hand we investigate conjunctions and disjunctions of arbitrary cardinality, and on the other hand we assume the peculiar axioms coming from~\cref{lem:lf}.
    \item[{\rm 2.}] Concerning regularity and exactness properties of $L\KFr_\lf$, it should be noted that  
    extensions of $S4$ where Craig's Theorem
    (i.e.\ regularity) holds are fully determined in~\cite{Copr}: the classification follows  Maksimova's parallel results for the finitary fragment~\cite{Mak79,Mak80} (with some notable exceptions, see~\cref{rem:grz} above). Passing to Barr-exactness, the picture changes drastically: in~\cite{Copr} it is shown that there are exacly five logics $L$ extending $S4$ and such that $L\KFr_\lf$ is Barr exact. These five logics are all locally finite and of little interest (the category of their locally finite Kripke frames turns out to be equivalent to the topos of the actions of a finite small monoid). Still, one may ask whether Barr exactness is so rare everywhere, in particular a classification over $K4$ cannot include only trivial examples, as witnessed by~\cref{ex:gllin} above.
\item[{\rm 3.}] Given that Barr exactness is rare, one may investigate \emph{almost Barr exactness}, where  a category is said to be almost Barr exact if it is regular and effective descent morphisms coincide with regular epis. Almost Barr exactness, in fact, turns out to be equivalent to effectiveness of special kinds of equivalence relations~\cite{Beyond}, so that there is a concrete chance that it should hold  for some relevant categories of locally finite Kripke frames. 
    \item[{\rm 4.}]  Finally, as pointed out in~\cite{MDB-SG}, $L\KFr_\lf$ 
    is a locally finitely presentable category
    (being equivalent to $\Lex(L\MA_\fin,\Set)$), hence it is the category of models of an essentially algebraic
first-order theory~\cite{AR}: this raises the problem of identifying and axiomatizing such a theory. A success in this attempt  would contribute to recast in an algebraic context some typically semantic notions. 
\end{enumerate}